\documentclass[amsthm]{elsart}      
\usepackage{yjsco}
\usepackage{natbib}
\newcommand{\deleted}[1]{}
\newcommand{\delete}[1]{}
\newcommand{\mynote}[1]{}
\newcommand\notes[1]{}
\usepackage{txfonts}
\usepackage{amssymb}
\usepackage{eucal}
\usepackage{graphicx}
\usepackage{amscd}
\usepackage[all]{xy}           
\usepackage[active]{srcltx} 

\usepackage{amsfonts,latexsym}
\usepackage{xspace}
\usepackage{epsfig}
\usepackage{float}
\usepackage{color}
\usepackage{fancybox}
\usepackage{colordvi}
\usepackage{multicol}
\usepackage[hypertex]{hyperref} 
\usepackage{stmaryrd}
\usepackage{amsthm}

\newcommand\mod{{\rm\ mod\,}}
\newcommand{\paren}[1]{$($#1$)$}
\renewcommand\bar[1]{\overline{#1}}
\renewcommand\tilde[1]{\widetilde{#1}}
\newcommand\text[1]{{\rm #1}}

\newcommand\changed[1]{#1}
\changed{}

\newtheorem{theorem}{Theorem}[section]
\newtheorem{exam}[theorem]{Example}
\newtheorem{lemma}[theorem]{Lemma}
\newtheorem{coro}[theorem]{Corollary}
\newtheorem{mremark}[theorem]{Remark}
\newtheorem{problem}[theorem]{Problem}
\newtheorem{conjecture}[theorem]{Conjecture}
\newtheorem{mprop}[theorem]{Proposition}
\newtheorem{mdefn}[theorem]{Definition}
\deleted{\newtheorem{theorem}{Theorem}[section]
\newtheorem{prop-def}{Proposition-Definition}[section]
\newtheorem{coro-def}{Corollary-Definition}[section]

}

\newcommand{\nc}{\newcommand}
\nc{\tred}[1]{\textcolor{red}{#1}}
\nc{\tblue}[1]{\textcolor{blue}{#1}}
\nc{\tgreen}[1]{\textcolor{green}{#1}}
\nc{\tpurple}[1]{\textcolor{purple}{#1}}
\nc{\btred}[1]{\textcolor{red}{\bf #1}}
\nc{\btblue}[1]{\textcolor{blue}{\bf #1}}
\nc{\btgreen}[1]{\textcolor{green}{\bf #1}}
\nc{\btpurple}[1]{\textcolor{purple}{\bf #1}}

\renewcommand{\Bbb}{\mathbb}
\newcommand\Complex{\Bbb C}

\newcommand{\efootnote}[1]{}

\renewcommand{\textbf}[1]{}
\nc{\mlabel}[1]{\label{#1}}  
\nc{\mcite}[1]{\citet{#1}}  
\nc{\mref}[1]{\ref{#1}}  
\nc{\mbibitem}[1]{\harvarditem{#1}} 

\delete{
\nc{\mlabel}[1]{\label{#1}}  
\nc{\mcite}[1]{\cite{#1}}  
\nc{\mref}[1]{\ref{#1}}  
\nc{\mbibitem}[1]{\bibitem{#1}} 
}

\delete{
\nc{\mlabel}[1]{\label{#1}  
{\hfill \hspace{1cm}{\bf{{\ }\hfill(#1)}}}}
\nc{\mcite}[1]{\cite{#1}{{\bf{{\ }(#1)}}}}  
\nc{\mref}[1]{\ref{#1}{{\bf{{\ }(#1)}}}}  
\nc{\mbibitem}[1]{\bibitem[\bf #1]{#1}} 
}

\renewcommand\geq{\geqslant}
\renewcommand\leq{\leqslant}

\delete{

\nc{\mlabel}[1]{\label{#1}}  
\nc{\mcite}[1]{\cite{#1}}  
\nc{\mref}[1]{\ref{#1}}  
\nc{\mbibitem}[1]{\bibitem{#1}} 
}

\nc{\lead}{\mathrm{Lead}}
\nc{\Id}{\mathrm{Id}}
\nc{\Irr}{\mathrm{Irr}}
\nc{\vx}{\sigma}
\nc{\vy}{\tau}
\nc{\dvx}{\sigma^{(1)}}
\nc{\dvy}{\tau^{(1)}}
\nc{\done}{\vep}

\nc{\wt}{\mathrm{wt}}
\nc{\bre}[1]{|#1|} \nc{\mapmonoid}{\frakM}
\nc{\disjoint}{\frakM'}
\nc{\ncpoly}[1]{\langle #1\rangle}  
\nc{\mapm}[1]{\lfloor\!|{#1}|\!\rfloor}             
\nc{\diff}[1]{{}^\NC\{ #1 \}}
\nc{\disj}[1]{\{{#1}\}'}
\nc{\mdisj}[1]{\frakM'(#1)}
\nc{\brho}{\bar{\rho}}
\nc{\om}{\bar{\frakm}}
\nc{\frakn}{\mathfrak n}
\nc{\ddeg}[1]{^{(#1)}}
\nc{\opset}{X} \nc{\genset}{{Z}}
\nc{\NC}{\mathrm{{NC}}}

\nc{\bin}[2]{ (_{\stackrel{\scs{#1}}{\scs{#2}}})}  
\nc{\binc}[2]{ \left (\!\! \begin{array}{c} \scs{#1}\\
    \scs{#2} \end{array}\!\! \right )}  
\nc{\bincc}[2]{  \left ( {\scs{#1} \atop
    \vspace{-1cm}\scs{#2}} \right )}  
\nc{\bs}{\bar{S}}
\nc{\cosum}{\sqsubset}
\nc{\la}{\longrightarrow}
\nc{\rar}{\rightarrow}
\nc{\dar}{\downarrow}
\nc{\dprod}{**}
\nc{\dap}[1]{\downarrow \rlap{$\scriptstyle{#1}$}}
\nc{\md}{\mathrm{dth}}
\nc{\uap}[1]{\uparrow \rlap{$\scriptstyle{#1}$}}
\nc{\defeq}{\stackrel{\rm def}{=}}
\nc{\disp}[1]{\displaystyle{#1}}
\nc{\dotcup}{\ \displaystyle{\bigcup^\bullet}\ }
\nc{\gzeta}{\bar{\zeta}}
\nc{\hcm}{\ \hat{,}\ }
\nc{\hts}{\hat{\otimes}}
\nc{\barot}{{\otimes}}
\nc{\free}[1]{\bar{#1}}
\nc{\uni}[1]{\tilde{#1}}
\nc{\hcirc}{\hat{\circ}}
\nc{\leng}{\ell}
\nc{\lleft}{[}
\nc{\lright}{]}
\nc{\lc}{\lfloor}
\nc{\rc}{\rfloor}
\nc{\lb}{[} 
\nc{\rb}{]} 
\nc{\curlyl}{\left \{ \begin{array}{c} {} \\ {} \end{array}
    \right .  \!\!\!\!\!\!\!}
\nc{\curlyr}{ \!\!\!\!\!\!\!
    \left . \begin{array}{c} {} \\ {} \end{array}
    \right \} }
\nc{\longmid}{\left | \begin{array}{c} {} \\ {} \end{array}
    \right . \!\!\!\!\!\!\!}
\nc{\onetree}{\bullet}
\nc{\ora}[1]{\stackrel{#1}{\rar}}
\nc{\ola}[1]{\stackrel{#1}{\la}}
\nc{\ot}{\otimes}
\nc{\mot}{{{\boxtimes\,}}}
\nc{\otm}{\overline{\boxtimes}}
\nc{\sprod}{\bullet}
\nc{\scs}[1]{\scriptstyle{#1}}
\nc{\mrm}[1]{{\rm #1}}
\nc{\msum}{\sum\limits}
\nc{\margin}[1]{\marginpar{\rm #1}}   
\nc{\dirlim}{\displaystyle{\lim_{\longrightarrow}}\,}
\nc{\invlim}{\displaystyle{\lim_{\longleftarrow}}\,}
\nc{\mvp}{\vspace{0.3cm}}
\nc{\tk}{^{(k)}}
\nc{\tp}{^\prime}
\nc{\ttp}{^{\prime\prime}}
\nc{\svp}{\vspace{2cm}}
\nc{\vp}{\vspace{8cm}}
\nc{\proofbegin}{\noindent{\bf Proof: }}
\nc{\proofend}{$\blacksquare$ \vspace{0.3cm}}
\nc{\modg}[1]{\!<\!\!{#1}\!\!>}
\nc{\intg}[1]{F_C(#1)}
\nc{\lmodg}{\!<\!\!}
\nc{\rmodg}{\!\!>\!}
\nc{\cpi}{\widehat{\Pi}}
\nc{\sha}{{\mbox{\cyr X}}}  
\nc{\shap}{{\mbox{\cyrs X}}} 
\nc{\shpr}{\diamond}    
\nc{\shp}{\ast}
\nc{\shplus}{\shpr^+}
\nc{\shprc}{\shpr_c}    
\nc{\msh}{\ast}
\nc{\zprod}{m_0}
\nc{\oprod}{m_1}
\nc{\vep}{\varepsilon}
\nc{\labs}{\mid\!}
\nc{\rabs}{\!\mid}

\nc{\dth}{d}
\nc{\mmbox}[1]{\mbox{\ #1\ }}
\nc{\fp}{\mrm{FP}} \nc{\rchar}{\mrm{char}} \nc{\Fil}{\mrm{Fil}}
\nc{\Mor}{Mor\xspace}
\nc{\gmzvs}{gMZV\xspace}
\nc{\gmzv}{gMZV\xspace}
\nc{\mzv}{MZV\xspace}
\nc{\mzvs}{MZVs\xspace}
\nc{\Hom}{\mrm{Hom}} \nc{\id}{\mrm{id}} \nc{\im}{\mrm{im}}
\nc{\incl}{\mrm{incl}} \nc{\map}{\mrm{Map}} \nc{\mchar}{\rm char}
\nc{\nz}{\rm NZ} \nc{\supp}{\mathrm Supp}

\nc{\Alg}{\mathbf{Alg}}
\nc{\Bax}{\mathbf{Bax}}
\nc{\bff}{\mathbf f}
\nc{\bfk}{{\bf k}}
\nc{\bfone}{{\bf 1}}
\nc{\bfx}{\mathbf x}
\nc{\bfy}{\mathbf y}
\nc{\base}[1]{\bfone^{\otimes ({#1}+1)}} 
\nc{\Cat}{\mathbf{Cat}}
\delete{}
\nc{\detail}{\marginpar{\bf More detail}
    \noindent{\bf Need more detail!}
    \svp}
\nc{\Int}{\mathbf{Int}}
\nc{\Mon}{\mathbf{Mon}}
\nc{\rbtm}{{shuffle }}
\nc{\rbto}{{Rota-Baxter }}
\nc{\remarks}{\noindent{\bf Remarks: }}
\nc{\Rings}{\mathbf{Rings}}
\nc{\Sets}{\mathbf{Sets}}

\nc{\BA}{{\Bbb A}} \nc{\CC}{{\Bbb C}} \nc{\DD}{{\Bbb D}}
\nc{\EE}{{\Bbb E}} \nc{\FF}{{\Bbb F}} \nc{\GG}{{\Bbb G}}
\nc{\HH}{{\Bbb H}} \nc{\LL}{{\Bbb L}} \nc{\NN}{{\Bbb N}}
\nc{\KK}{{\Bbb K}} \nc{\QQ}{{\Bbb Q}} \nc{\RR}{{\Bbb R}}
\nc{\TT}{{\Bbb T}} \nc{\VV}{{\Bbb V}} \nc{\ZZ}{{\Bbb Z}}

\nc{\cala}{{\mathcal A}} \nc{\calc}{{\mathcal C}}
\nc{\cald}{{\mathcal D}} \nc{\cale}{{\mathcal E}}
\nc{\calf}{{\mathcal F}} \nc{\calg}{{\mathcal G}}
\nc{\calh}{{\mathcal H}} \nc{\cali}{{\mathcal I}}
\nc{\call}{{\mathcal L}} \nc{\calm}{{\mathcal M}}
\nc{\caln}{{\mathcal N}} \nc{\calo}{{\mathcal O}}
\nc{\calp}{{\mathcal P}} \nc{\calr}{{\mathcal R}}
\nc{\cals}{{\mathcal S}}
\nc{\calt}{{\mathcal T}} \nc{\calw}{{\mathcal W}}
\nc{\calk}{{\mathcal K}} \nc{\calx}{{\mathcal X}}
\nc{\CA}{\mathcal{A}}

\nc{\fraka}{{\mathfrak a}} \nc{\frakA}{{\mathfrak A}}
\nc{\frakb}{{\mathfrak b}} \nc{\frakB}{{\mathfrak B}}
\nc{\frakD}{{\mathfrak D}} \nc{\frakH}{{\mathfrak H}}
\nc{\frakM}{{\mathfrak M}} \nc{\bfrakM}{\overline{\frakM}}
\nc{\frakm}{{\mathfrak m}} \nc{\frakP}{{\mathfrak P}}
\nc{\frakN}{{\mathfrak N}} \nc{\frakp}{{\mathfrak p}}
\nc{\frakS}{{\mathfrak S}} \nc{\frakx}{{\mathfrak x}}
\nc{\ox}{\bar{\frakx}} \nc{\frakX}{{\mathfrak X}}
\nc{\fraky}{{\mathfrak y}} \nc\dop{\delta}

\font\cyr=wncyr10
\font\cyrs=wncyr7
\nc{\redt}[1]{\textcolor{red}{#1}}

\nc{\li}[1]{\textcolor{red}{\tt Li:#1}}
\nc{\nli}[1]{\textcolor{red}{\tt Li:#1}}
\nc{\ws}[1]{\textcolor{blue}{William: #1}}
\nc{\rh}[1]{\textcolor{blue}{Ronghua:#1}}
\nc{\rhz}[1]{\textcolor{red}{Ronghua:#1}}

\journal{Journal of Symbolic Computation}

\begin{document}
\hyphenpenalty=8000
\begin{frontmatter}

\notes{\ead[url]{URL 1}}

\title{Differential Type Operators and\\Gr\"obner-Shirshov Bases}
%
%
\author{Li Guo}
\address{Department of Mathematics and Computer Science,
         Rutgers University,
         Newark, NJ 07102, USA}
\ead{liguo@newark.rutgers.edu}
\author{William Y. Sit}
\address{Dept. of Math., The City College of The City
University of New York, New York, NY 10031, USA}
\ead{wyscc@sci.ccny.cuny.edu}
\author{Ronghua Zhang}
\address{Research Institute of Natural Sciences, Yunnan University,
Kunming 650091, China}
\ead{rhzhang@ynu.edu.cn}


\date{\today}

\begin{abstract}
A long standing problem of Gian-Carlo Rota for associative algebras
is the classification of all linear operators that can be defined on
them. In the 1970s, there were only a few known operators, for
example, the derivative operator, the difference operator, the
average operator, and the Rota-Baxter operator. A few more appeared
after Rota posed his problem. However, little progress was made to
solve this problem in general. In part, this is because the precise
meaning of the problem is not so well understood. In this paper, we
propose a formulation of the problem using the framework of operated
algebras and viewing an associative algebra with a linear operator
as one that satisfies a certain operated polynomial identity. This
framework also allows us to apply theories of rewriting systems and
Gr\"{o}bner-Shirshov bases. To narrow our focus more on the
operators that Rota was interested in, we further consider two
particular classes of operators, namely, those that generalize
differential or Rota-Baxter operators. As it turns out, these two
classes of operators correspond to those that possess
Gr\"obner-Shirshov bases under two different monomial orderings.
Working in this framework, and with the aid of computer algebra, we
are able to come up with a list of these two classes of operators,
and provide some evidence that these lists may be complete. Our
search has revealed quite a few new operators of these types whose
properties are expected to be similar to the differential operator
and Rota-Baxter operator respectively.

Recently, a more unified approach has emerged in related areas, such
as difference algebra and differential algebra, and Rota-Baxter
algebra and Nijenhuis algebra. The similarities in these theories
can be more efficiently explored by advances on Rota's problem.
\end{abstract}

\begin{keyword}
Rota's Problem; rewriting systems, Gr\"obner-Shirshov bases;
operators; classification; differential type operators, Rota-Baxter
type operators.
\end{keyword}

\end{frontmatter}


\setcounter{section}{0}


\section{Introduction}
\mlabel{sec:int}

Throughout the history of mathematics, objects are often understood
by studying operators defined on them. Well-known examples are found
in Galois theory, where a field is studied by its automorphisms, and
in analysis and geometry, where functions and manifolds are studied
through derivatives and vector fields. These operators abstract to
the following linear operators on associative algebras.
\vspace{-12pt}
\begin{eqnarray}
\hspace{0.6in}\text{automorphism}\hspace{1in}
P(x\,y)& = &P(x\,)P(y), \mlabel{eq:hom}\\
\text{derivation}\hspace{1.02in}  \dop(x\,y)& = &\dop(x)\,y+x\,
\dop(y). \mlabel{eq:der}
\end{eqnarray}
By the 1970s, several more special operators, denoted by $P$ below
with corresponding name and defining property, had been studied in
analysis, probability and combinatorics, including, for a fixed
constant $\lambda$,
\vspace{-12pt}
\begin{eqnarray}\text{average}\hspace{1in}  P(x)\,P(y)& =& P(x\,P(y)),
\mlabel{eq:av}\\
\text{inverse\ average}\hspace{1in} P(x)\,P(y)& = & P(P(x)\,y),
\mlabel{eq:iav}
\\
\text{(Rota-)Baxter\ (weight\ \lambda)}\hspace{1in} P(x)\,P(y)& =
&P(x\,P(y)+P(x)\,y+\lambda\, x\,y),
\mlabel{eq:rb} \\
\text{Reynolds}\hspace{1in}P(x)\,P(y)& =
&P(x\,P(y)+P(x)\,y-P(x)\,P(y)). \mlabel{eq:rey}
\end{eqnarray}
\mcite{Ro2} posed the question of finding all the identities that
could be satisfied by a linear operator defined on associative
algebras. He also suggested that there should not be many such
operators other than these previously known ones.\footnote{The
following is quoted from Rota's paper. ``In a series of papers, I
have tried to show that other linear operators satisfying algebraic
identities may be of equal importance in studying certain algebraic
phenomena, and I have posed the problem of finding all possible
algebraic identities that can be satisfied by a linear operator on
an algebra. Simple computations show that the possibility are very
few, and the problem of classifying all such identities is very
probably completely solvable. A notable step forward has been made
in the unpublished (and unsubmitted) Harvard thesis of Alexander
Doohovskoy.'' He also remarked that a partial (but fairly complete)
list of such identities are Eq.~(\mref{eq:hom})-(\mref{eq:rey}).}
Even though there was some work on relating these different
operators~\citep{Fr}, little progress was made on finding {\it all}
such operators. In the meantime, new identities for operators have
emerged from physics, algebra and combinatorial studies, such as
\vspace{-12pt}
\begin{eqnarray}
\text{Nijenhuis}\hspace{0.8in}
P(x)\,P(y)&=&P(x\,P(y)+P(x)\,y-P(x\,y)), \mlabel{eq:ni}
\\
\text{Leroux's\ TD} \hspace{0.8in} P(x)\,P(y)& = &P(x\,P(y)
+P(x)\,y-x\,P(1)\,y),\mlabel{eq:td}\\
\hspace{0.3in}\text{derivation\ (weight\ \lambda)}\hspace{1in}
\dop(x\,y)&=&\dop(x)\,y+x\,\dop(y)+\lambda\,\dop(x)\,\dop(y).
\mlabel{eq:dlambda}
\end{eqnarray}
The previously known operators continue to find remarkable
applications in pure and applied mathematics. For differential
operators, we have the development of differential algebra
\citep{Kol}, difference algebra~\citep{Co,AL}, and quantum
differential operators~\citep{LR, LR2}. For Rota-Baxter algebras, we
note their relationship with the classical Yang-Baxter equation,
operads, combinatorics, and most prominently, the renormalization of
quantum field theory through the Hopf algebra framework of Connes
and Kreimer
\citep{C-K1,G-K1,G-K3,Ag,AGKO,E-G-K3,EGM,G-S,Bai,E-G,G-Z}.

\subsection{Our approach}
These interesting developments motivate us to return to Rota's
question and try to understand the problem
better.\footnote{Disclaimer: We are still exploring the best way to
formulate Rota's problem and nothing in this paper is meant to
provide a definitive formulation.} In doing so, we found that two
key points in Rota's question deserve further thoughts. First, we
need a suitable framework to formulate precisely what is an
``operator identity,'' and second, we need to determine key
properties that characterize the classes of operator identities
that are of interest to other areas of mathematics, such as those
listed above.

For the first point, we note that a simplified but analogous
framework has already been formulated in the 1960s and subsequently
explored with great success. This is the study of PI-rings and
PI-algebras, whose elements satisfy a set of polynomial identities,
or PIs for short ~\citep{Pr,Ro,DF}.

Let $\bfk$ be a commutative unitary ring. In this paper, all
algebras are unitary, associative $\bfk$-algebras that are generally
non-commutative, and all algebra homomorphisms will be over $\bfk$,
unless the contrary is noted or obvious.

Recall that an algebra $R$ {\it satisfies a polynomial identity} if
there is a non-zero (non-commutative) polynomial $\phi(X)$ in a
finite set $X$ of indeterminates over $\bfk$ (that is, $\phi(X) \in
\bfk\ncpoly{X}$, the free algebra on $X$) such that $\phi(X)$ is
sent to zero under any algebra homomorphism $f: \bfk\ncpoly{X} \to
R$. To generalize this framework to the operator case, we shall
introduce formally in Section~\mref{sec:form} the notion of operated
algebras and the construction of the free operated algebra
$\bfk\mapm{X}$ on $X$, which shall henceforth be called {\it the
operated polynomial algebra on $X$}.  An operator identity will
correspond to a particular element $\phi(X)$ in $\bfk\mapm{X}$.
Analogous to PI-algebras, an OPI-algebra $R$ is an algebra with a
$\bfk$-linear operator $P$, a finite set $X$, and an operated
polynomial $\phi(X) \in \bfk\mapm{X}$ such that $\phi(X)$ is sent to
zero under any morphism (of operated algebras) $f:\bfk\mapm{X} \to
R$. The operated polynomial $\phi$, or the equation $\phi(X) = 0$,
is called an operated polynomial identity (OPI) on $R$ and we say
$P$ (as well as $R$) {\it satisfies the OPI} $\phi$ (or
$\phi(X)=0$).

As a first example, a differential algebra\footnote{We illustrate
only with an ordinary differential algebra, where the common
notation for the derivation is $\delta$. In this paper, we have three symbols for
the operator: $\lc\, \rc$, $P$, and $\delta$, to be used
respectively for a general (or bracketed word) setting, the
Rota-Baxter setting, and the differential/difference setting; often,
they are interchangeable. We use $\lc\,\rc$ for $\bfk\mapm{X}$ to
emphasize that $\bfk\mapm{X}$ is {\it not} the differential
polynomial ring. Any dependence of the operator on parameters is
suppressed, unless clarity requires otherwise. \mlabel{ft:one}} is
an OPI-algebra $R$ with operator $\dop$, where the OPI is defined
using $X=\{x,y\}$ and $\phi(x,y):=\lc xy \rc -\lc x \rc y-x \lc
y\rc$, where $\lc\,\rc$ denotes the operator in $\bfk\mapm{X} =
\bfk\mapm{x,y}$. As a second example, a difference algebra $S$ is an
OPI-algebra where the $\bfk$-linear operator $P$ is an endomorphism,
that is, $(S, P)$ satisfies $P(r)P(s) = P(rs)$ for all $r,s \in S$.
A common difference algebra (taken from \cite[pp.\,104--5]{AL}) is
the following: Let $z_0 \in \Complex$, where $\Complex$ is the field
of complex numbers, and let $S$ be the field of all functions $f(z)$
of one complex variable $z$ meromorphic in the region $U = \{ z \in
\Complex \mid ({\rm Re}\,z)({\rm Re}\,z_0) \geq 0\}$ (so that $z+z_0
\in U$ for all $z \in U$), then the shift (or translation) operator
$P$ taking $f(z) \in S$ to $f(z+z_0) \in S$ is an automorphism of
$S$, making $(S, P)$ an (inversive) difference algebra.

With all operator identities understood to be OPIs in
$\bfk\mapm{X}$, the second point mentioned above may at first be
interpreted as follows: among all OPIs, which ones are particularly
consistent with the associative algebra structure so that they are
singled out for study?
This is a
subtle question since one might argue (correctly, see Proposition
\ref{pp:frpio}) that any OPI defines a class of (perhaps trivial)
operated algebras, just like any PI defines a class of algebras. We
approach this by making use of two related theories: rewriting
systems and Gr\"{o}bner-Shirshov bases.

First, we shall regard an OPI as a rule that defines a rewriting
system\footnote{We remind the reader that a term rewriting rule is a
one-way replacement rule that depends on a term-order, unlike an
equality or a congruence.} and study certain properties of this
rewriting system, such as termination and confluence, that will
characterize OPIs of interest. Termination and confluence are
essential and desirable properties since we discovered our lists of
OPIs by symbolic computation. As a rewriting rule, an OPI $\phi$ can
be applied recursively and if not carefully done, such applications
may lead to infinite recursion, in which case, it is no longer
computationally feasible to derive meaningful consequences on the
associative algebra from the OPI $\phi$. An example is the Reynolds
operator identity in Eq.\,(\mref{eq:rey}), where, if taken as a
rewriting rule by replacing the equal sign with $\to$, the right
hand side contains the expression $P(x)P(y)$, which equals the
left-hand-side, leading to more and more complicated expressions as
the rewriting rule is applied repeatedly {\it ad infinitum}.

By putting aside for now OPIs like the Reynolds identity, we in
effect restrict the class of OPIs under investigation and this
allows us to apply symbolic computation to search for a list of
identities for two broad families that include all the (other)
previously mentioned OPIs. One family of operators consists of the
OPIs of differential type, which include derivations, endomorphisms,
differential operators of weight $\lambda$, and more generally
operators $\delta$ satisfying an OPI of the form $\phi := \lc xy \rc
- N(x,y)$, where $N(x,y)$ is a formal expression in $\bfk\mapm{x,y}$
in {\it differentially reduced form}, that is, it does not contain
any subexpression of the form $\lc uv \rc$ for any $u, v \in
\bfk\mapm{x,y}$. The other family consists of the OPIs of
Rota-Baxter type, which include those defining the average,
Rota-Baxter, Nijenhuis, Leroux's TD operators, and more generally
OPIs of the form $\phi := \lc x\rc \lc y \rc-\lc M(x,y)\rc$ where
$M(x,y)$ is an expression in $\bfk\mapm{x,y}$ in {\it Rota-Baxter
reduced form}, that is, it does not involve any subexpression of the
form $\lc u \rc \lc v\rc$ for any $u, v \in
\bfk\mapm{x,y}$.\footnote{This, by definition, excludes the Reynolds
operator as it stands. However, if we rewrite the Reynolds identity
as $P(P(x)P(y)) = P(xP(y))+ P(P(x)y) - P(x)P(y)$, then it would be
computationally feasible to explore its interaction with
associativity, and would suggest that the Reynolds operator belongs
to a ``higher order'' class.}

These two families share a common feature: each OPI involves a
product: $xy$ for differential type, and $\lc x\rc \lc y \rc$ for
Rota-Baxter type. These families of OPIs thus provide properties
arising from the associativity of multiplication, which we can
explore in our computational experiments. More generally, for an OPI
that gives rise to a terminating rewriting system, the associative
law imposes various confluence constraints that may be satisfied by
some operated algebras, but not by others. Thus, another advantage
of the rewriting system approach is that we may use such constraints
as criteria to screen OPI-algebras for further research.

In Section~\mref{sec:form} of this paper, we begin the construction
of the free operated algebras $\bfk\mapm{X}$ using a basis of
bracketed words in $X$. This will be the universal space for OPIs by
which we formulate Rota's problem precisely in a general setting of
a free operated algebra satisfying an OPI $\phi$. In
Section~\mref{sec:GS}, we develop Gr\"obner-Shirshov bases for free
operated algebras and prove the Composition-Diamond Lemma (Theorem
\ref{thm:CDL}). In Section~\mref{sec:diff}, we define operators and
operated algebras of differential type and propose a conjectural
answer to Rota's Problem in this case with a list of differential
type OPIs. As evidence of our conjecture, we verify in Section
\mref{ss:evid} that the operators in our list all satisfy the
properties prescribed for a differential type operator, and in
Section~\mref{ss:diffGS}, we prove several equivalent criteria for
an OPI $\phi$ in $\bfk\mapm{x,y}$ in differentially reduced form to
be of differential type (Theorem \ref{thm:gsdiff}), a result that
connects together the rewriting system induced by $\phi$, the
Gr\"obner-Shirshov bases of the operated ideal induced by $\phi$,
and the free operated algebras satisfying $\phi$. In
Section~\mref{sec:RB}, we define similarly operators and operated
algebras of Rota-Baxter type and give a conjecture for the complete
list of Rota-Baxter type OPIs. In Section~\mref{sec:comp}, we give a
description of an empirical {\it Mathematica} program by which we
obtained the lists. In the final Section~\mref{sec:sum}, we
explain our approach in the context of varieties of algebras,
providing research directions towards a further understanding of
Rota's Problem, leading possibly to new tools and theoretical proofs
of our conjectures.

\section{Operator identities}
\mlabel{sec:form}

In this section we give a precise definition of an OPI in the
framework of operated algebras.\footnote{The concepts, construction
of free objects and results in this section are covered in more
generality in texts on universal algebra \citep{BS,Cohn,BN}. Our
review makes this paper more accessible and allows us to establish
our own notations.} We review the concept of operated (associative)
monoids, operated algebras, and bracketed words, followed by a
construction for the free operated monoids and algebras using
bracketed words. Bracketed words are related to Motzkin words and
decorated rooted trees~\citep{Guop}.

\subsection{Operated monoids and algebras}
\begin{mdefn}
{\rm An {\bf operated monoid}\footnote{As remarked in Footnote
\ref{ft:one}, we use the same symbol $P$ for all distinguished maps
and hence we shall simply use $U$ for an operated monoid. In this
paper, all semigroups and monoids are associative but generally
non-commutative.} is a monoid $U$ together with a map $P: U\to U$. A
morphism from an operated monoid\, $U$ to an operated monoid $V$ is
a monoid homomorphism $f :U\to V$ such that $f \circ P= P \circ f$,
that is, the diagram below is commutative: {$$
 \xymatrix{
            U\ar[rr]^P \ar[d]^f && U \ar[d]_f \\
            V\ar[rr]^P          && V}
$$ }}\mlabel{de:mapset}
\end{mdefn}
Let $\bfk$ be a commutative unitary ring. In Definition
\mref{de:mapset}, we may replace ``monoid'' by ``semigroup,''
``$\bfk$-algebra,'' or ``nonunitary $\bfk$-algebra'' to
define\footnote{To adapt Definition \mref{de:mapset} for operated
$\bfk$-algebra categories, $P$ is assumed to be a $\bfk$-linear map
and $f$ is a morphism of the underlying $\bfk$-algebras.} {\bf
operated semigroup}, {\bf operated $\bfk$-algebra} and {\bf operated
nonunitary $\bfk$-algebra}, respectively. For example, the semigroup
$\calf$ of rooted forests, with the concatenation product and the
grafting map $\lc\;\rc$, turns $\calf$ into an operated
semigroup~\citep{Guop}. The $\bfk$-module $\bfk\,\calf$ generated by
$\calf$ is an operated nonunitary $\bfk$-algebra. The unitarization
of this algebra has appeared in the work of \mcite{C-K0} on
renormalization of quantum field theory.

The adjoint functor of the forgetful functor from the category of
operated monoids to the category of sets gives the free operated
monoids in the usual way.  More precisely, a {\bf free operated
monoid} on a set $X$ is an operated monoid $U$ together with a map
$j:X\to U$ with the property that, for any operated monoid $V$
together with a map $f:X\to V$, there is a unique morphism
$\free{f}:U\to V$ of operated monoids such that $f=\free{f}\circ j.$
Any two free operated monoid on the same set $X$ are isomorphic via
a unique isomorphism. 

We similarly define the notion of a free operated (nonunitary)
$\bfk$-algebra on a set $X$. As shown in~\mcite{Guop}, the
operated non-unitary $\bfk$-algebra of rooted forests mentioned
above is the free operated non-unitary $\bfk$-algebra on one
generator.

An {\bf operated ideal} in an operated $\bfk$-algebra $R$ is an
ideal closed under the operator. The operated ideal {\bf generated
by a set} $\Phi \subseteq R$ is the smallest operated ideal in $R$
containing $\Phi$.

\subsection{Free operated monoids}

For any set $Y$, let $M(Y)$ be the free monoid generated by $Y$ and
let $\lc Y\rc$ be the set $\{ \lc y\rc \mid y\in Y\}$, which is just
another copy of $Y$ whose elements are denoted by $\lc y\rc$ for
distinction.

We now construct the free operated monoid over a given set $X$ as
the limit of a directed system $$\{\,\iota_{n}: \mapmonoid_n\to
\mapmonoid_{n+1}\, \}_{n=0}^\infty$$ of free monoids $\mapmonoid_n$,
where the transition morphisms $\iota_{n}$ will be natural
embeddings. For this purpose, let $\mapmonoid_0=M(X)$, and let
$$ \mapmonoid_1:=M(X\cup \lc \mapmonoid_0\rc).$$ Let $\iota_{0}$
be the natural embedding $\iota_{0}:\mapmonoid_0 \hookrightarrow
    \mapmonoid_1$.
Note that elements in $\lc M(X)\rc$ are only symbols indexed by
elements in $M(X)$. Thus, while $\bfone\in \mapmonoid_0$ is
identified with $\iota_{0}(\bfone)=\bfone\in \mapmonoid_1$, $\lc
\bfone\rc \in \mapmonoid_1$ is not the identity.

Assuming by induction that for some $n\geq 2$, we have defined the
free monoids $\mapmonoid_i, 0\leq i\leq n-1,$ and the embedding
 $ \iota_{i-2}: \mapmonoid_{i-2} \to \mapmonoid_{i-1}, 0\leq i\leq n-2.$
 Let
\begin{equation}
 \mapmonoid_n(X):=M(X\cup \lc\mapmonoid_{n-1}\rc ).
 \mlabel{eq:frakm}
 \end{equation}
The identity map on $X$ and the embedding $\iota_{n-2}$ induce an
injection
\begin{equation}
 \iota_{n-1}: X\cup \lc\mapmonoid_{n-2}\rc \hookrightarrow
    X\cup \lc \mapmonoid_{n-1} \rc,
\mlabel{eq:transet}
\end{equation}
which, by the functoriality of $M$, extends to an embedding (still
denoted by $\iota_{n-1}$) of free monoids
\begin{equation}
 \iota_{n-1}: \mapmonoid_{n-1} = M(X\cup \lc\mapmonoid_{n-2}\rc)\hookrightarrow
    M(X\cup \lc \mapmonoid_{n-1}\rc) = \mapmonoid_{n}. \mlabel{eq:tranm}
\end{equation}
\noindent This completes the construction of the directed system.
Finally we define the monoid\footnote{We adopt two notations for the
free operated monoid on $X$. The notation $\mapm{X}$, suggested by a
reviewer, is simpler and more natural, but $\mapmonoid(X)$ is
consistent with prior literature and occasionally, typographically
more pleasing, as in $\lc \mapmonoid(X) \rc$, when compared to $\lc
\mapm{X} \rc$.} $\mapmonoid(X)$ by
$$
\mapm{X} = \mapmonoid(X):=\dirlim \mapmonoid_n = \bigcup_{n\geq 0}
 \mapmonoid_n,
$$ where the identity of $\mapm{X}$ is (the
directed limit of) $\bfone$.

\begin{theorem} {\bf (\mcite{Guop}, Corollaries 3.6 and 3.7)}{\sl
\begin{enumerate}
\renewcommand\labelenumi{{\rm $(\theenumi)$}}
\item The monoid
$\mapm{X}$, with operator $P := \lc\;\rc$ and natural embedding
$j:X\to \mapm{X}$, is the free operated monoid on $X$.
\mlabel{it:mapsetm}
\item The unitary \paren{associative} $\bfk$-algebra
     $\bfk\mapm{X}$, with the $\bfk$-linear operator $P$ induced
     by $\lc\;\rc$ and the natural embedding $j:X
\to \bfk\mapm{X}$, is the free
 operated unitary $\bfk$-algebra on $X$.
    \mlabel{it:mapalgm}
\end{enumerate}
\mlabel{thm:freetm} }
\end{theorem}

\begin{mdefn} An element $w \in \mapm{X}$ is called a {\bf
bracketed word on the generator set} $X$. If $X = \{x_1, \dots,
x_k\}$, we also write $\bfk\mapm{X}$ simply as $\bfk\mapm{x_1,
\dots, x_k}$. An element $\phi \in \bfk\mapm{X}$ but not in $\bfk$
is called a {\bf bracketed polynomial in} $X$. \mlabel{def:bracket}
\end{mdefn}

A nonunit element $w$ of $\mapm{X}= \mapmonoid(X)$ can be uniquely
expressed in the form
\begin{equation}
 w=w_1\cdots w_k \quad \mmbox{for some} k \mmbox{and some} w_i\in
 X\cup \lc \mapmonoid(X)\rc, 1\leq i\leq k.
\mlabel{eq:stde}
\end{equation}

\begin{mdefn}
For a nonunit element $w \in \mapm{X}=\mapmonoid(X)$, the
decomposition in Eq.\,(\mref{eq:stde}) is called the {\bf standard
decomposition} of $w$ and elements in $X\cup \lc \mapmonoid(X)\rc$
are called {\bf indecomposable}. The integer $\bre{w}:=k$ is called
the {\bf breadth} of $w$. The integer $d(w):=\min\{ n \,|\, w\in
\mapmonoid_n\}$ is called the {\bf depth} of $w$. We also consider
$\bfone$ (the empty product in $\mapm{X}$ and Eq.\,(\mref{eq:stde}))
to be indecomposable and define $\bre(\bfone) = d(\bfone) = 0$.
\mlabel{def:decomp}
\end{mdefn}

\begin{mremark}
Alternatively~\citep{Guop}, $\mapm{X}$ can be viewed as the set of
{\bf bracketed words} $w$ of the free monoid $M(X\cup \{\lc, \rc\})$
generated by $X\cup \{\lc, \rc\}$, in which the brackets $\lc\;\rc$
form balanced pairs, or more explicitly,
\begin{enumerate}
\item the total number of $\lc$ in the word $w$ equals to the total
number of $\rc$ in $w$; and
\item counting from the left to the right of $w$,
the number of $\lc$ is always greater than or equal to the number of
$\rc$.
\end{enumerate}
\mlabel{rk:string}
\end{mremark}
For example, for the set $X=\{x\}$, the element $w :=\lc x\rc x\lc
x\lc x\rc\rc$ is a bracketed word in $M(\{x,\lc,\rc\})$, with
$\bre{w} = 3$ and $d(w)=2$, while neither $\lc \lc x\rc$ (failing
the first condition) nor $\rc x \lc$ (failing the second condition)
is.

\subsection{Operated polynomial identity algebras}
We recall the concept of a polynomial identity algebra. Let
$\bfk\langle X \rangle$ be the free non-commutative
$\bfk$-algebra on a finite set $X= \{x_1, \dots, x_k\}$. A given
$\phi \in \bfk\ncpoly{X}$, $\phi \neq 0$, defines a category
$\Alg_\phi$ of algebras, whose objects are $\bfk$-algebras $R$
satisfying $ \phi(r_1,\dots,r_k)=0$ for all $r_1,\dots,r_k\in R$.
The non-commutative polynomial $\phi$ (formally, the equation
$\phi(x_1, \dots, x_k) = 0$, or its equivalent $\phi_1(x_1, \dots,
x_k) = \phi_2(x_1, \dots, x_k)$ if $\phi := \phi_1 - \phi_2$) is
classically called a {\bf polynomial identity} (PI) and we say $R$
is a {\bf PI-algebra} if $R$ {\bf satisfies} $\phi$ for some $\phi$.
For any set $Z$, we may define the free PI-algebra on $Z$ in
$\Alg_\phi$ by the obvious universal property.

We extend this notion to operated algebras. Let $\phi\in
\bfk\mapm{x_1,\cdots,x_k}$, let $R$ be an operated algebra, and let $r = (r_1,
\dots, r_k) \in R^k$. The substitution map $f_r:
\{x_1,\dots,x_k\}\to R$ that maps $x_i$ to $r_i$ induces a unique
morphism $\free{f_r}:\bfk\mapm{x_1,\dots,x_k}\to R$ of operated
algebras that extends $f_r$.\mynote{ by the universal property of
$\bfk\mapm{x_1,\dots,x_k}$ as a free operated algebra on
$\{x_1,\dots,x_k\}$.}  Let $\phi_R:R^k \to R$ be defined by
\begin{equation}{\phi}_R(r_1,\dots,r_k) :=
\free{f_r}(\phi(x_1,\dots,x_k)).\mlabel{eq:phibar}\end{equation}

\begin{mdefn}
{\rm Let $\phi\in \bfk\mapm{x_1,\cdots,x_k}$ and $R$ be an operated algebra. If
$${\phi}_R(r_1,\dots,r_k)=0, \quad \forall\ r_1,\dots,r_k\in R,$$
then $R$ is called a {\bf $\phi$-algebra}, the operator $P$ defining
$R$ is called a {\bf $\phi$-operator}, and $\phi$
\paren{or $\phi=0$} is called an {\bf operated polynomial identity}
\paren{OPI}.
 An {\bf operated polynomial identity algebra} or an {\bf OPI-algebra} is a
$\phi$-algebra for some $\phi \in \bfk\mapm{x_1,\dots,x_k}$\, for
some positive integer $k$. } \mlabel{de:pio}
\end{mdefn}

\begin{exam}\mlabel{ex:opi-a}
When $\phi:=\lc xy\rc -x\lc y\rc -\lc x\rc y$, then a
$\phi$-operator on a $\bfk$-algebra $R$ is a derivation on $R$,
usually denoted by $\delta$, and $R$ is an ordinary, possibly
non-commutative, differential algebra in which $\delta(a) = 0$ for
all $a \in \bfk$. 
\end{exam}

\begin{exam}\mlabel{ex:opi-b}
When $\phi:=\lc x\rc \lc y\rc -\lc x\lc y\rc \rc -\lc \lc x\rc y\rc
-\lambda\lc xy\rc $, where $\lambda \in \bfk$, then a
$\phi$-operator \paren{resp. $\phi$-algebra} is a Rota-Baxter
operator \paren{resp. Rota-Baxter algebra} of weight $\lambda$. We
denote such operators by $P$.
\end{exam}

\begin{exam}\mlabel{ex:opi-c}
When $\phi$ is from the noncommutative polynomial algebra
$\bfk\langle X\rangle$, then a $\phi$-algebra is an algebra with
polynomial identity, which we may view as an operated algebra where
the operator is the identity map.
\end{exam}

The next proposition is a consequence of the universal property of
free operated algebras and can be regarded as a special case of a
very general result on $\Omega$-algebras, where $\Omega$ is a set
called the signature and $\Omega$ represents a family of operations
on the algebra \citep[see e.g.][Chapter\,I, Proposition\,3.6]{Cohn}.
We caution the reader that there are two sets involved: the set $X$
in terms of which an OPI is expressed, and the set $Z$ on which the
free $\phi$-algebra is constructed.

\begin{mprop} {\bf \paren{\citeauthor{BN} \paren{\citeyear{BN}},
Theorem~3.5.6}} {\sl\   Let $\genset$ be a set, let $R =
\bfk\mapm{\genset}$, and let $j_\genset: \genset \to R$ be the
natural embedding. Let $X=\{x_1,\dots,x_k\}$ and $\phi \in
\bfk\mapm{X}$. Let ${\phi}_R: R^k \to R$ be as defined in {\em
Eq.\,(\mref{eq:phibar})}, let $I_\phi(\genset)$ be the operated
ideal of $R$ generated by the set
$$
\left\{\,{\phi}_R(r_1,\dots,r_k)\ |\ r_1,\dots,r_k\in R\,\right\},$$
and let $\pi_\phi: R \to
R/I_\phi(\genset)$ be the quotient morphism. Let
 $$ i_\genset:=\pi_\phi\circ j_\genset: \genset \to
 R/I_\phi(\genset).
 $$
 Then the quotient operated algebra
$R/I_\phi(\genset)$, together with $i_\genset$ and the operator $P$
induced by $\lc\,\rc$, is the free $\phi$-algebra on $\genset$.
\mlabel{pp:frpio}}
\end{mprop}

For a specific proof of Proposition \mref{pp:frpio}, see
\citet{GSZ}. Proposition \mref{pp:frpio} shows that for any non-zero
$\phi \in \bfk\mapm{X}$, there is always a (free, associative, but
perhaps trivial) $\phi$-algebra. Thus the ``formulation'' below of
Rota's Problem would not be helpful.

\medskip
\begin{quote}
Find all non-zero $\phi\in \bfk\mapm{X}$ such that the OPI $\phi=0$
can be satisfied by {\it some} linear\\operator on {\it some}
associative algebra. \mlabel{pr:rota}
\end{quote}

While the construction in Proposition \mref{pp:frpio} is general, we
note that the free $\phi$-algebra may have hidden consequences.

\begin{exam}\mlabel{ex:freak}
Let $\phi(x,y) := \lc xy\rc - y\lc x\rc$. Let $\genset$ be a set and
let $Q = \bfk\mapm{\genset}/I_\phi(\genset)$ be the free
$\phi$-algebra with the operator $P$ induced by $\lc\,\rc$ on
$R=\bfk\mapm{\genset}$. Let $a,b,c \in Q$ be arbitrary. We must have
$P((ab)c)=P(a(bc))$. Applying the identity $\phi = 0$ on $Q$ to both
sides, we find that $(bc -cb)P(a) = 0$. We do not know
if $I_\phi(\genset)$ is completely prime\footnote{Recall that there are two notions of primeness for an ideal $I$ of a not-necessarily commutative ring $R$: $I$ ($I \ne R$) is {\it completely prime} if $uv\in I$ for $u, v\in R$ implies that either $u\in I$ or $v\in I$; and  $I$ is {\it prime} if for any ideals U and V, $UV \subseteq I$ implies either $U \subseteq I$ or $V \subseteq I$. When $R$ is commutative, the two definitions are equivalent.} or not, but if it is, then we would
have two possibilities: $Q$ is commutative, or $Q$ is not
commutative but $P(a)=0$ for all $a \in Q$. We also note that any
commutative algebra with the identity as operator is a
$\phi$-algebra.
\end{exam}

\section{Gr\"obner-Shirshov bases for free operated algebras}
\mlabel{sec:GS}

We now introduce the framework of Gr\"obner-Shirshov bases for the
free operated algebra $\bfk\mapm{X}$ on $X$. Shirshov basis was
first studied by \cite{Zh} and then by \cite{Sh-a,Sh-b}. For a
historic review, we refer the reader to the Introduction and
Bibliography sections of \cite{BCQ}, who gave a good survey of
methods to construct linear bases, and in particular,
Gr\"obner-Shirshov bases, in algebras under various combinations of
commutativity and associativity. \mcite{DK} has further details on
the relationship
 of Gr\"obner-Shirshov bases with the well-known
work of \cite{Buch} and \cite{Be}. We also provided a sketchy
summary in \cite{GS}. Recently, these bases have been obtained
by~\mcite{BCQ} for free {\it nonunitary} operated algebras. We will
consider the case of free {\it unitary} operated algebras.

With the notation in~\mcite{BCQ}, let $\bfk\langle X;\Omega\rangle$
denote the free {\em nonunitary} associative algebra on $X$ with a
set $\Omega$ of linear operators. When $\Omega$ consists only of one
unary operator $\lc\,\rc$, $\bfk\langle X;\Omega\rangle$ is the
non-unitary version of $\bfk\mapm{X}$ and may be constructed as
$\bfk\frakS$, where $$\frakS=\dirlim \frakS_n $$

\vspace{-7pt} \noindent with $\frakS_n$ defined recursively by
$$ \frakS_n:= S(X\cup \lc \frakS_{n-1}\rc), \quad \frakS_0:=S(X),
$$
and where, for any set $Y$, $S(Y)$ is the semigroup generated by
$Y$.

As is well-known, the difference between an associative algebra $A$
and its unitarization $\uni{A}$ is very simple: $\uni{A}=A\oplus
\bfk\,1$. For an operated algebra, the difference is much more
significant, as we can already see from their constructions. Since
we are studying operators on unitary algebras, we need to be careful
adapting results from~\mcite{BCQ}. For this reason and for
introducing notation, we establish here some results that
will lead to the Composition-Diamond Lemma (Theorem \ref{thm:CDL})
and construction of Gr\"obner-Shirshov bases for {\em unitary}
operated algebras.

\newcommand\mapmXstar{\mapm{X}^\star}
\vspace{-2pt}
\begin{mdefn} Let $\star$ be a symbol not in $X$ and
let $X^\star = X \cup \{\star\}$. By a {\bf $\star$-bracketed word
on} $X$, we mean any expression in $\mapm{X^\star}$ with exactly one
occurrence of $\star$. The set of all $\star$-bracketed words on $X$
is denoted by $\mapm{X}^\star$. For $q\in \mapmXstar$ and $u \in
 \mapm{X}$, we define
 \vspace{-5pt}
\begin{equation}
q|_{u}:=q|_{\star\mapsto u},\mlabel{eq:eva}
\end{equation}
\vspace{-5pt}\noindent to be the bracketed word obtained by
replacing the letter $\star$ in $q$ by $u$, and call $q|_u$ a
$u$-{\bf bracketed word on} $X$. Further, for $s=\sum_i c_i u_i\in
\bfk\mapm{X}$, where $c_i\in\bfk$ and $u_i\in \mapm{X}$, and $q \in
\mapmXstar$, we define \vspace{-5pt}
\begin{equation}
 q|_{s}:=\sum_i c_i q|_{u_i}\,,
\mlabel{eq:evasum}
\end{equation}
\vspace{-5pt}\noindent and extend by linearity to define the symbol
$q|_s$ for any $q \in \bfk\mapmXstar$. {Note that $q|_s$ is in
general not a bracketed word but a bracketed polynomial.}
\end{mdefn}
\vspace{-2pt} This process is the same as the process of replacing
subterms in \citeauthor{BN} (\citeyear[Definition\,3.1.3]{BN}). We
note the following simple relationship between operator replacement
and ideal generation.

\begin{lemma}{\sl
Let $S$ be a subset of\, $\bfk\mapm{X}$. Let\, $\Id(S)$ be the
operated ideal of $\bfk\mapm{X}$ generated by $S$. Then
$$\Id(S)=\left\{ \sum_{i=1}^k c_i q_i|_{s_i}\,\Big|\ c_i\in \bfk,
q_i\in \mapmXstar, s_i\in S, 1\leq i\leq k, k\geq 1\right\}.$$
\mlabel{lem:repgen}}
\end{lemma}

\begin{proof}
It is clear that the right hand side is contained in the left hand
side. On the other hand, the right hand side is already an operated
ideal of $\bfk\mapm{X}$ that contains $S$.
\end{proof}

\newcommand\mapmXss{\mapm{X}^{\star_1,\star_2}}
\begin{mdefn} For distinct symbols $\star_1,\star_2$
not in $X$, let $X^{\star_1,\star_2} = X \cup \{\star_1, \star_2\}$.
We define a {\bf $(\star_1, \star_2)$-bracketed word on $X$} to be a
bracketed word in $\mapm{X^{\star_1,\star_2}}$ with exactly one
occurrence of $\star_1$ and exactly one occurrence of $\star_2$. The
set of $(\star_1, \star_2)$-bracketed words on $X$ is denoted by
$\mapmXss$. For $q\in \mapmXss$ and $u_1,u_2\in \bfk\mapm{X}$, we
define \vspace{-6pt}
\begin{equation}
q|_{u_1,\ u_2}= q|_{\star_1\mapsto u_1,\star_2\mapsto u_2},
\mlabel{eq:2star1}
\end{equation}
\vspace{-6pt} to be the bracketed word obtained by
replacing the letter $\star_1$ (resp. $\star_2$) in $q$ by $u_1$ (resp. $u_2$) and call it a {\bf $(u_1, u_2)$-bracketed word on
$X$}.\end{mdefn} A $(u_1, u_2)$-bracketed word on $X$ can also be
recursively defined by \vspace{-6pt}
\begin{equation}
q|_{u_1,u_2}:=(q^{\star_1}|_{u_1})|_{u_2}, \mlabel{eq:2star2}
\end{equation}
\vspace{-6pt} where $q^{\star_1}$ is $q$ when $q$ is regarded as a
$\star_1$-bracketed word on the set $X^{\star_2}$. Then
$q^{\star_1}|_{u_1}$ is in $\mapm{X}^{\star_2}$ and we can apply
Eq.~(\mref{eq:eva}). Similarly, treating $q$ first as a
$\star_2$-bracketed word $q^{\star_2}$ on the set $X^{\star_1}$, we
have \vspace{-6pt}
\begin{equation}
q|_{u_1,u_2}:=(q^{\star_2}|_{u_2})|_{u_1}. \mlabel{eq:2star3}
\end{equation}

\begin{mdefn} A {\bf monomial ordering} on $\mapm{X}$ is a
well-ordering $\leq$ on $\mapm{X}$ satisfying the two conditions:
\vspace{-6pt}\begin{equation} \bfone \leq u; \quad  u < v
\Rightarrow q|_u < q|_v, \quad \text{\ for\ all\ } u, v \in \mapm{X}
\text{\ and\ all\ } q \in \mapmXstar.\mlabel{eq:ordering}
\end{equation}
\vspace{-6pt}
Here, as usual, we denote $u<v$ if $u\leq v$ but $u\neq v$. Given a monomial ordering $\leq$ and a bracketed
polynomial $s\in \bfk\mapm{X}$, we let $\bar{s}$ denote the leading
bracketed word (monomial) of $s$. If the coefficient of $\bar{s}$ in
$s$ is $1$, we call $s$ {\bf monic with respect to the monomial
order $\leq$}\,.
\end{mdefn}

Examples of such orderings will be considered later in this
paper.
For now, we fix a monomial ordering $\leq$ on $\mapm{X}$.

\begin{lemma} {\sl Let $s, s'\in \bfk\mapm{X}$, let $t\in \mapm{X}$
and suppose $\,\bar{s}<t$. Then
\smallskip
\begin{enumerate}
\renewcommand\labelenumi{{\rm $(\theenumi)$}}
\item
For any $q \in \mapmXstar$, we have $\
\overline{q|_s}=q|_{\overline{s}} < q|_t$\,. \mlabel{it:lead1}
\smallskip \item For $q \in \mapmXss$, we have
$$\overline{q|_{s,\,s'}} = \overline{q|_{\bar{s},\,s'}}
=\overline{q|_{s,\,\overline{s'}}} = q|_{\bar{s},\, \overline{s'}} <
q|_{t,\, \overline{s'}} \quad \text{\ and\ }\quad
\overline{q|_{s',\, s}}<q|_{\overline{s'},\, t}\,.
\mlabel{eq:lead}$$ \mlabel{it:lead2} \mlabel{it:lead3}
\end{enumerate}
\mlabel{lem:lead}}
\end{lemma}

\vspace{-20pt}
\begin{proof}
(\mref{it:lead1}) Let $s=\sum_{i=1}^k a_i s_i$ where $0\neq a_i\in
\bfk$ and $s_i\in \mapm{X}$ with $s_1>\cdots>s_k$. Thus
$\overline{s}=s_1$.  By definition, $ q|_s=\sum_{i=1}^k a_i
q|_{s_i}$ and by Eq.\,(\mref{eq:ordering}) for a monomial order,
$q|_{s_1}>\cdots
> q|_{s_k}$. Thus $\overline{q|_s} = q|_{s_1} = q|_{\overline{s}}.$
The inequality follows by the property in Eq.\,(\mref{eq:ordering}) of
a monomial order.

\medskip \noindent (\mref{it:lead2}) Let $s=\sum_{i=1}^k a_i s_i\in
\bfk\mapm{X}$ be as in Part (\mref{it:lead1}). Thus $\bar{s} = s_1$.
By Eq.\,(\mref{eq:2star2}) and Part (\mref{it:lead1}), we have
\begin{equation}
\overline{q|_{s_i,\,s'}} = \overline{(q^{\star_1}|_{s_i})|_{s'}}=
(q^{\star_1}|_{s_i})|_{\overline{s'}}= q_{s_i,\,\overline{s'}},
\qquad (1 \leq i \leq k).\mlabel{eq:lead4a} \end{equation}

\noindent By Eq.\,(\mref{eq:2star3}) and the property in
Eq.\,(\mref{eq:ordering}) of a monomial order, we have \vspace{-6pt}
\begin{equation} q|_{s_1,\,\overline{s'}} =
(q^{\star_2}|_{\overline{s'}})|_{s_1} > \cdots
> (q^{\star_2}|_{\overline{s'}})|_{s_k} = q|_{s_k,\,\overline{s'}}.
\mlabel{eq:lead4b}\end{equation} \vspace{-6pt}\noindent It follows
from Eqs.\,(\mref{eq:lead4a}) and (\mref{eq:lead4b}) that
$\overline{q|_{s_1,\,s'}} > \cdots > \overline{q|_{s_k,\,s'}}$ and
since by linearity, $q|_{s,\, s'} = \sum_{i=1}^k a_i q|_{s_i,\,
s'}$, we have \vspace{-6pt}\begin{equation}\overline{q|_{s,\,s'}} =
\max\,\{ \overline{q|_{s_i,\,s'}} \mid 1 \leq i \leq k\,\} =
\overline{q|_{s_1,\,s'}} =q|_{s_1,\,\overline{s'}} =
q|_{\bar{s},\,\overline{s'}}.\mlabel{eq:lead4c}\end{equation}\vspace{-6pt}
\noindent The first equality in Part (\mref{it:lead2}) follows by
replacing $s$ with $\bar{s}$ in Eq.\,(\mref{eq:lead4c}), and the
second by replacing $s'$ with $\overline{s'}$. By the equalities
just proved and Eq.\,(\mref{eq:2star3}), we have
$$\overline{q|_{s,\,s'}} = q|_{\bar{s}, \overline{s'}} =
(q^{\star_2}|_{\overline{s'}})|_{\bar{s}} <
(q^{\star_2}|_{\overline{s'}})|_t = q|_{t,\, \overline{s'}}.$$
\noindent The other inequality follows similarly.
\end{proof}

The following concepts of intersection and including compositions
are adapted from~\mcite{BCQ}. For operated algebras, they are
analogous to the concepts of overlap and inclusion $S$-polynomials
for associative algebras, as in~\mcite{Be}. Here we pay careful
attention to ensure these concepts are well-defined.

\begin{mdefn}
{\rm Let $f, g \in
\bfk\mapm{X}$ be two bracketed polynomials monic with respect to
$\leq$.
\begin{enumerate}
\setlength\itemsep{3pt}
\item \label{def:int}
If there exist $\mu,\nu, w\in \mapm{X}$ such that
$w=\bar{f}\mu=\nu\bar{g}$ with
$\bre{w}< \bre{\bar{f}}+\bre{\bar{g}}$, then we define
$$(f,g)_{w}:=(f,g)^{\mu,\nu}_w:=f\mu-\nu g$$
and call it the {\bf intersection composition} of $f$ and $g$ with
respect to $(\mu,\nu)$.
\item \label{def:inc} If there exist a $q\in \mapmXstar$ and $w \in \mapm{X}$
such that $w =\bar{f}=q|_{\bar{g}}$,
then we define
$$(f,g)^{q}_w:=f-q|_{g}$$ and call it an {\bf including
composition of $f$ and $g$ with respect to $q$}.
\end{enumerate}
} \mlabel{de:comp}
\end{mdefn}

\begin{mremark}\label{ex:inc}
{We note that the superscripts $\mu,\nu$ for the intersection composition
$(f,g)^{\mu,\nu}_w$ is not necessary, since $\mu$ and $\nu$ are uniquely
defined by $w$, indeed, by $\bre{w}$, because of the uniqueness of the standard decompositions of $\bar{f}, \mu, \nu, \bar{g}$. However,
the superscript $q$ in the including
composition $(f,g)_w^q$ is needed to ensure that the notation is well-defined.
For example, if $\bar{g}$
occurs in $\bar{f}$ more than once, we might have two different
$q$'s that give the same $q|_{\bar{g}}$ but different including
compositions. To illustrate, take $f=xyx, g=x-1\in \bfk\mapm{x,y}$ and
$q_1=\star yx, q_2=xy \star\in \mapmXstar$. Then we have
$$xyx=w=\bar{f}=q_1|_{\bar{g}}=\bar{f}=q_2|_{\bar{g}}.$$ But
$(f,g)_w^{q_1}=f-q_1|_{g}=yx$ and $(f,g)_w^{q_2}=f-q_2|_g=xy$ are
not the same.}
\end{mremark}

\begin{mremark}
If Definition \ref{de:comp}(\ref{def:int}) holds with $\mu=1$, then
the intersection composition is also an including composition. For
if $\bar{f}=\nu\bar{g}$, then $\bar{f}=q|_{\bar{g}}$ where
$q=\nu\star$. Hence $(f,g)_{\bar{f}}^{1,\nu}=(f,g)_{\bar{f}}^q\,.$
However, if $\nu=1$ but $\mu \ne 1$, then since $\overline{f}\mu =
\overline{g}$, there is no $q \in \mapm{X}^\star$ satisfying
$\overline{f} = q|_{\overline{g}}$. As should have been clear,
Definition \ref{de:comp}(\ref{def:inc}) is not symmetric with
respect to $f$ and $g$.
\end{mremark}

\begin{mdefn}
{\rm Let $S$ be a set of monic bracketed polynomials and let $w\in
\mapm{X}$.
\begin{enumerate}
\item
For $u, v\in \bfk\mapm{X}$, we call $u$ and $v$ {\bf congruent
modulo $(S,w)$} and denote this by
$$ u\equiv v \ \mod (S,w) $$
if $u-v=\sum_ic_iq_i|_{s_i}, $
 with $c_i\in \bfk$,  $q_i\in \mapmXstar$, $s_i\in S$
and $q_i|_{\overline{s_i}}< w$.
\item
For $f, g\in \bfk\mapm{X}$ and suitable $w, \mu,\nu$ or $q$ that give
an intersection composition $(f,g)_w^{\mu,\nu}$ or an including
composition $(f,g)_w^q$, the composition is called {\bf trivial
modulo $(S,w)$} if
$$
(f,g)_w^{\mu,\nu} \text{\ or\ } (f,g)_w^q\equiv 0 \ \mod (S,w).
$$
\item
The set $S \subseteq \bfk\mapm{X}$ is a {\bf
Gr\"{o}bner-Shirshov basis} if, for all $f,g\in S$, all
intersection compositions $(f,g)_w^{\mu,\nu}$ and all including
compositions $(f,g)_w^q$ are trivial modulo $(S,w)$.
\end{enumerate}
} \mlabel{de:GS}
\end{mdefn}

\begin{mdefn} \begin{enumerate}
\item
Let $u, w$ be two  bracketed words in $\mapm{X}$. We call $u$ a {\bf
subword} of $w$ if $w$ is in the operated ideal of $\mapm{X}$
generated by $u$. In terms of $\star$-words, $u$ is a subword of $w$
if there is a $q\in \mapmXstar$ such that $w=q|_u$.
{A subword $u$ of $w$
is a subword when viewed as a string in the free monoid $M(X\cup\{\lc, \rc\})$ as in
Remark~\mref{rk:string}:
namely the string of letters forming $u$ is a substring of the string of
letters forming $w$.}
\item
Let $u_1$ and $u_2$ be two subwords of $w$.
\begin{enumerate}
\item $u_1$ and $u_2$ are called {\bf separated} if there is $q\in
\mapmXss$ such that $w=q|_{u_1,u_2}$. In terms of strings in
$M(X\cup\{\lc, \rc\})$, this means that the substrings $u_1$ and
$u_2$ of $w$ have no overlap.
\item $u_1$ and $u_2$ are called {\bf overlapping} if there are
subwords $a, b, c$ of $w$ such that $au_1=c=u_2b$ or $au_2=c=u_1b$ with $\bre{c}<\bre{u_1}+\bre{u_2}$.
In terms of strings in $M(X\cup\{\lc, \rc\})$, this means that the
strings of $u_1$ and $u_2$ have an overlap.
\end{enumerate}
\end{enumerate}\mlabel{def:subwords}
\end{mdefn}
We note there is a third relative location of $u_1$ and $u_2$ in $w$, namely either $u_1$ or $u_2$ is nested in (i.e., a subword of) the other.

\begin{mprop} \mlabel{prop:separate} Let $S \subseteq \bfk\mapm{X}$.
{\sl Let $s_1,s_2 \in S$ and suppose there exist $q_1, q_2 \in
\mapmXstar$ and $w \in \mapm{X}$ such that $w =
q_1|_{\overline{s_1}} = q_2|_{\overline{s_2}}$, by which we may view
$\overline{s_1}, \overline{s_2}$ as subwords of $w$ and suppose as
such, $\overline{s_1}$ and $\overline{s_2}$ are separated in $w$.
Then $q_1|_{s_1} \equiv q_2|_{s_2} \mod(S,w).$}
\end{mprop}

\begin{proof} Let $q\in
\mapmXss$ be the $(\star_1,\star_2)$-bracketed word obtained by
replacing {\it this} occurrence of $\overline{s_1}$  in $w$ by
$\star_1$ and {\it this} occurrence of $\overline{s_2}$ in $w$ by
$\star_2$. Then we have
$$q^{\star_1}|_{\overline{s_1}} = q_2, \quad
q^{\star_2}|_{\overline{s_2}} = q_1, \quad {\rm\ and\ }\quad
q|_{\overline{s_1},\
\overline{s_2}}=q_1|_{\overline{s_1}}=q_2|_{\overline{s_2}} =w,
$$
where in the first two equalities, we have identified
$\mapm{X}^{\star_2}$ and $\mapm{X}^{\star_1}$ with
$\mapmXstar$. Let $$s_1=\overline{s_1}+\sum_i c_{i}u_{i}, \qquad s_2=\overline{s_2}+\sum_j d_{j}v_{j}$$ where $c_{i}, d_{j}\in \bfk$,
$u_{i}, v_{j}\in \mapm{X}$, and
$\overline{u_i} = u_{i}<\overline{s_1}$ and $\overline{v_j}=
v_{j}< \overline{s_2}$. Then by the linearity of $s_1,s_2$ in $q|_{s_1,s_2}$,
we have
\vspace{-20pt}
\begin{eqnarray*} q_2|_{ s_2}-q_1|_{
s_1}&=&(q^{\star_1}|_{\overline{s_1}})|_{s_2}
-(q^{\star_2}|_{\overline{s_2}})|_{s_1}\\
&=&
q|_{\overline{s_1}, \ s_2}-q|_{s_1, \ \overline{s_2}}\\
&=&(-q|_{-\overline{s_1}, \ s_2}+q|_{s_1, \
s_2})+(q|_{s_1,\ -\overline{s_2}}-q|_{s_1, \ s_2})\\
&=&-q|_{s_1-\overline{s_1}, \ s_2}+q|_{s_1,\
s_2-\overline{s_2}}\\
&=&-(q^{\star_1}|_{s_1-\overline{s_1}})|_{s_2} +(
q^{\star_2}|_{s_2-\overline{s_2}})|_{s_1}\\
&=&-\sum_i c_i(q^{\star_1}|_{u_{i}})|_{s_2} +\sum_j
d_{j}(q^{\star_2}|_{v_{j}})|_{s_1},
\end{eqnarray*}
{Since $u_i = \overline{u_i}$ and $v_j = \overline{v_j}$, by Eqs.\,(\ref{eq:2star2}) and (\ref{eq:2star3}), we have
$$(q^{\star_1}|_{u_{i}})|_{\overline{s_2}} =q_{\overline{u_{i}},\,
\overline{s_2}}<q_{\overline{s_1},\,\overline{s_2}}=w \text{\ and\ }
(q^{\star_2}|_{v_{j}})|_{\overline{s_1}} =q_{\overline{s_1},
\,\overline{v_{j}}}<q_{\overline{s_1},\,\overline{s_2}}=w.$$
}
\noindent
This means
$$
q_1|_{s_1}\equiv q_2|_{s_2} \mod (S,w),
$$
completing the proof.
\end{proof}

\begin{lemma}{\sl Let $\leq$ be a monomial ordering of $\bfk\mapm{X}$ and
let $S$ be a set of monic bracketed polynomials in $\bfk\mapm{X}$.
Then the following conditions on $S$ are equivalent:
\begin{enumerate}
\renewcommand\labelenumi{{\rm $(\theenumi)$}}
\item \mlabel{it:Scong1} $S$ is a Gr\"{o}bner-Shirshov basis.
\item \mlabel{it:Scong2} For every $s_1, s_2 \in S$ and $w \in \mapm{X}$ for which there
exist $q_1, q_2 \in \mapm{X}^{\star}$  such
that $w = q_1|_{\overline{s_1}}= q_2|_{\overline{s_2}}$, we have
$q_1|_{s_1} \equiv q_2|_{s_2} \mod (S, w)$. 
\end{enumerate} \mlabel{lem:Scong}}
\end{lemma}

\begin{proof}(\mref{it:Scong2}) $\Rightarrow$ (\mref{it:Scong1}):
This is clear since the congruences include those from intersection
composition and inclusion composition.

(\mref{it:Scong1}) $\Rightarrow$ (\mref{it:Scong2}): Let $s_1, s_2
\in S$ and $w \in \mapm{X}$, and suppose there exist $q_1, q_2 \in \mapm{X}^{\star}$ such that $w = q_1|_{\overline{s_1}}=
q_2|_{\overline{s_2}}$. We fix one such occurrence of $\overline{s_1}$ and one such occurrence of $\overline{s_2}$. We distinguish three cases according to the
relative location of these particular occurrences of $\overline{s_1}$ and $\overline{s_2}$ in $w$.

\footnote{We note that there might be multiple occurrences of $\overline{s_1}$ and/or $\overline{s_2}$ in $w$, with different relative locations. If so, then we need to consider each of them separately. For example, take $s_1=ab$, $s_2=bc$ and $w=abcabc$. Then $\overline{s_1}=ab$ and $\overline{s_2}=ba$ both appear twice in $w$, as shown below.

$$w=\underbrace{ab}_{1}\,c\underbrace{ab}_{2}\,c =a\overbrace{bc}^1a\overbrace{bc}^2.$$
Then we need to consider the four (pairs of) occurrences of $\overline{s_1}$ and $\overline{s_2}$ in $w$, two of which are separated and two of which overlap. }

\noindent {\bf Case I.} Suppose the bracketed words
$\overline{s_1}$ and $\overline{s_2}$ are separated in $w$. This
case is Proposition \mref{prop:separate}.
\smallskip

\noindent {\bf Case II.}  Suppose the bracketed words
$\overline{s_1}$ and $\overline{s_2}$ overlap in $w$. Then by switching $s_1$ and $s_2$ if necessary, we might assume that there exist some bracketed subwords $w_1, \mu, \nu, \in \mapm{X}$ of $w$ such that
$w_1=\overline{s_1}\mu=\nu\overline{s_2}$  with $\bre{w_1}<\bre{\,\overline{s_1}\,} +\bre{\,\overline{s_2}\,}$. Thus there is $p\in \mapm{X}^\star$ such that $p|_{w_1}=p|_{\overline{s_1}\mu}=w$ and then $q_1=p|_{\star \mu}$. Let $q:=p|_{\star_1\star_2}\in \mapmXss$ be obtained from $q_1$ by replacing $\star$
by $\star_1$ and $\mu$ by $\star_2$. Then we have
$$
q^{\star_2}|_\mu = q_1, \quad q^{\star_1}|_\nu = q_2, \quad {\rm\
and\ }\quad p|_{\overline{s_1}\mu} =q|_{\overline{s_1},\,\mu} =q_1|_{\overline{s_{_1}}}
=w.
$$
where in the first two equalities, we have identified
$\mapm{X}^{\star_2}$ and $\mapm{X}^{\star_1}$ with
$\mapmXstar$. Thus, we have
$$
q_1|_{ s_{_1}}- q_{_2}|_{ s_{_2}}=(q^{\star_2}|_\mu)|_{s_1} -
(q^{\star_1}|_\nu)|_ {s_2}= p|_{s_1\mu-\nu s_2}.
$$
Since  $S $ is a Gr\"{o}bner-Shirshov basis, we have
$$
s_1\mu-\nu s_2=\sum c_jp_j|_{t_j},
$$
where each $ c_j\in \bfk, \ p_j\in \mapmXstar, \ t_j\in S$ and
$p_j|_{\overline{t_j}}< w_1$.  By linearity,
$$q_1|_{ s_{_1}}- q_{_2}|_{ s_{_2}}= p|_{s_1\mu-\nu s_2} = \sum c_j
p|_{p_j|_{t_j}}.$$
By Lemma~\mref{lem:lead}.\mref{it:lead1}, $\overline{p_j|_{t_j}} =p_j|_{\overline{t_j}}<w_1$.
Thus
$$ \overline{p|_{p_j|_{t_j}}} =p|_{\overline{p_j|_{t_j}}}<p|_{w_1} =p|_{\overline{s_1}\mu} =q_1|_{\overline{s_1}}=w.$$
Therefore
$$
q_1|_{s_1}\equiv q_2|_{s_2} \ \mod (S,w).
$$

\smallskip
\noindent {\bf Case III.} Suppose one of the bracketed words
$\overline{s_1}$,
 $\overline{s_2}$ is a subword of the other.
Without loss of generality, suppose
$\overline{s_1}=q|_{\overline{s_2}}$ for some $\star$-bracketed word
$q$.  Then we have an inclusion composition
$(s_1,s_2)^q_{\overline{s_1}}\,.$

Since $S$ is a Gr\"{o}bner-Shirshov basis, we have
$$(s_1,s_2)^q_{\overline{s_1}}=s_1-q|_{s_2} =\sum_j c_jp_j|_{t_j},$$
with $ c_j\in \bfk, \ p_j\in \mapmXstar, \ t_j\in S$ and
$p_j|_{\overline{t_j}}< \overline{s_1}$. Then
\begin{equation}w =
q_2|_{\overline{s_2}} = q_1|_{\overline{s_1}}=
q_1|_{\overline{q|_{s_2}}} = q_1|_{q|_{\overline{s_2}}} =
p|_{\overline{s_2}},\mlabel{eq:compatible}\end{equation} where $p
\in \mapmXstar$ is obtained from $q_1$ by replacing $\star$ with
$q$.

Now $S$ is a Gr\"obner-Shirshov basis.  Hence we may write, by
Case II that has been proved and in which we take $q_1=p$ and $s_1=s_2$,
$$p|_{s_2} - q_2|_{s_2} = \sum_i d_i r_i|_{v_i}$$
where $d_i \in \bfk$, $r_i \in \mapmXstar$ and $v_i \in S$ and
$\overline{r_i|_{v_i}} < w$. Hence \begin{eqnarray*} q_2|_{s_2} -
q_1|_{s_1} &=& p|_{s_2} - \sum_i
d_i r_i|_{v_i} - q_1|_{s_1}\\
&=&q_1|_{q|_{s_2}}-q_1|_{s_1} - \sum_i d_i r_i|_{v_i}\\
&=& -q_1|_{s_1 - q|_{s_2}} - \sum_i d_i r_i|_{v_i}\\
&=& -\sum_j c_j q_1|_{p_j|_{t_j}} - \sum_i d_i r_i|_{v_i}\\
&=& -\sum_j c_j (q_1|_{p_j})|_{t_j} - \sum_i d_i r_i|_{v_i}.
\end{eqnarray*}

Now $t_j$ is in $S$ and
$$(q_1|_{p_j})|_{\overline{t_j}} =
q_1|_{p_j|_{\overline{t_j}}}<q_1|_{\overline{s_1}} = w.$$  Thus, we
obtain $q_2|_{s_2}-q_1|_{s_1}\equiv 0 \mod (S,w).$
\end{proof}

The following version of Composition-Diamond lemma can also be
proved by the same argument as its nonunitary
analogue~\cite[Theorem~3.2]{BCQ}.

\begin{theorem}{\rm $($Composition-Diamond lemma$)$} {\sl Let $S$ be a set
of monic bracketed polynomials in $\bfk\mapm{X}$, $>$ a monomial
ordering on $\mapm{X}$ and $\Id(S)$ the operated ideal of
$\bfk\mapm{X}$ generated by $S$.  Then the following statements are
equivalent:
\begin{enumerate}
\item $S $ is a Gr\"{o}bner-Shirshov basis in $\bfk\mapm{X}$.
\mlabel{it:cd1}
\item If $f\neq 0$ is in $\Id(S)$, then $\bar{f}=q|_{\overline{s}}$
for some $q \in \mapmXstar$ and $s\in S$. \mlabel{it:cd2}
\item  If $f\neq 0$ is in $\Id(S)$, then
\begin{equation}
f= c_1q_1|_{s_1}+ c_2q_2|_{s_2}+\cdots+ c_nq_n|_{s_n},
\label{eq:fexp1}
\end{equation}
where $ c_i\in \bfk, \ s_i\in S,\ q_i\in \mapmXstar,
q_1|_{\overline{s_1}}>q_2|_{\overline{s_2}}
>\cdots>q_n|_{\overline{s_n}}.$ \mlabel{it:cd3}
\item $\bfk\mapm{Z}=\bfk  \Irr(S)\oplus \Id(S)$ where
$$ \Irr(S): = \mapm{X}\backslash \left\{\,q|_{\overline{s}} \,|\,  q \in \mapmXstar, s\in S\right\}$$
and\, $\Irr(S)$ is a $\bfk$-basis of $\bfk\mapm{X}/\Id(S)$. \mlabel{it:cd4}
\end{enumerate}
\mlabel{thm:CDL} }
\end{theorem}

Before providing its proof, we give the following immediate corollary of the theorem.

\begin{coro}
Let $I$ be an operated ideal of $\bfk\mapm{X}$. If $I$ has a
generating set $S$ that is a Gr\"obner-Shirshov basis, then $
\Irr(S)$ is a $\bfk$-basis of $\bfk\mapm{X}/I$. \mlabel{co:GS}
\end{coro}

\begin{proof}
 (\mref{it:cd1}) $\Longrightarrow$ (\mref{it:cd2})
Let  $0\neq f\in \Id(S)$. Then by Lemma~\mref{lem:repgen}, $f$ is of the form

\begin{equation}
f=\sum\limits_{i=1}^{k} c_i q_i|_{ s_i},
\quad
0\neq  c_i\in \bfk, q_i\in \mathfrak{M}^\star(X), s_i\in S, 1\leq i\leq k.
\mlabel{eq:fexp}
\end{equation}
Let $w_i=q_i|_{\overline{ s_i}}$. We rearrange them in
non-increasing order by
$$
w_1= w_2=\cdots=w_m >w_{m+1}\geq \cdots\geq w_k.
$$

If for each $0\neq f\in \Id(S)$ there is a choice of the above sum such that $m=1$, then $\bar{f}=q_1|_{\overline{s_1}}$ and we are done. So suppose the implication (\mref{it:cd1}) $\Longrightarrow$ (\mref{it:cd2}) does not hold. Then there is $0\neq f\in \Id(S)$ such that for any expression in Eq.~(\mref{eq:fexp}), we have $m\geq 2$. Fix such an $f$ and choose an expression in Eq.~(\mref{eq:fexp}) such that $w_1=q_1|{\over{s_1}}$ is minimal and such that $m$ is minimal for this choice of $w_1$, that is, with the fewest $q_i|_{s_i}$ such that $q_i|_{\overline{s_i}}=q_1|_{\overline{s_1}}$. Since $m\geq 2$, we have
$q_1|_{\overline{s_1}}=w_1=w_2=q_2|_{\overline{s_2}}$.

Since $S $ is a Gr\"{o}bner-Shirshov basis in $\bfk\mapm{X}$, by
Lemma~\mref{lem:Scong}, we have
$$
q_2|_{ s_2}-q_1|_{ s_1}=\sum_j d_jp_j|_{r_j}
$$
where $d_j\in \bfk, \ r_j\in S, p_j\in \mapmXstar$  and
$p_j|_{\overline{r_j}}<w_1$. Thus
\begin{eqnarray*}
f&=& \sum_{i=1}^k  c_iq_i|_{s_i} \\
&=& ( c_1+ c_2)q_1|_{s_1}+ c_3q_3|_{s_3}+\cdots
+ c_mq_m|_{s_m} + \sum_{i=m+1}^k  c_i q_i|_{s_i}
+ \sum_j c_2d_j p_j|_{r_j}.
\end{eqnarray*}

By the minimality of $m$, we must have $c_1+c_2=c_3=\cdots=c_m=0$. We then obtain an expression of $f$ in the form of Eq.~(\mref{eq:fexp}) for which $q_1|_{\overline{s_1}}$ is even smaller. This is a contradiction.

\smallskip

\noindent (\mref{it:cd2}) $\Longrightarrow$ (\mref{it:cd3}). Suppose
the implication does not hold. Let $F$ be the set of
counterexamples, namely those $0\neq f\in \Id(S)$ that cannot be
written in the form of Eq.~(\mref{eq:fexp1}). Then the set
$\{\bar{f}\,|\,f\in F\}$ of leading terms is not empty. Then there
is an $f$ such that $\bar{f}$ is minimal in this set. By
Item~(\mref{it:cd2}), there are $q\in \mapmXstar$ and $s\in S$ such
that $\bar{f}=q|_{\bar{s}}$. Since $f$ is in $F$ and $q|_s$ is not
in $F$, $f-q|_s$ is not zero. But
$\bar{f}-\overline{q|_s}=\bar{f}-q|_{\bar{s}}=0$ means that
$\bar{f-q|_s}$ is less than $\bar{f}$. By the minimality of $f$ in
$F$, $f-q|_s\neq 0$ is not in $F$ and hence can be written in the
form of Eq.~(\mref{eq:fexp1}). But this means that $f$ can also be
written in the form of Eq.~(\mref{eq:fexp1}). This is a contradiction.

\smallskip

\noindent (\mref{it:cd3}) $\Longrightarrow$ (\mref{it:cd4}).
Obviously $0\in \bfk\, \Irr(S)+\Id(S)\subseteq \bfk\mapm{X}$. Suppose the
inclusion is proper. Then $\bfk\mapm{X}\backslash (\bfk
 \Irr(S)+\Id(S))$ contains only nonzero elements. Let $f\in \bfk\mapm{Z}\backslash (\bfk
 \Irr(S)+\Id(S))$ be such that
\begin{equation}
\bar{f}=\min\left\{\,\bar{g}\,|\, g\in \bfk\mapm{X}\backslash
(\bfk  \Irr(S)+\Id(S))\,\right\}.
\mlabel{eq:F}
\end{equation}
Suppose $\bar{f}$ is in $ \Irr(S)$, then $f\neq \bar{f}$ since $f\not\in
 \Irr(S)$. So $0\neq f-\bar{f}$ is in $\bfk\mapm{Z}\backslash
(\bfk  \Irr(S)+\Id(S))$ with $\overline{f-\bar{f}}<\bar{f}$. This is a contradiction. But suppose $\bar{f}$ is not in $\Irr(S)$. Then $\bar{f}=q|_{\bar{s}}$ for some $q\in \mapm{Z}^\star$ and $s\in S$. Then $\overline{f-q|_s}<\bar{f}$. If $f=q|_s$, then $f$ is in $\Id(S)$, a contradiction. Thus $f\neq q|_s$. Then $0\neq f-q|_s$ with $\overline{f-q|_s}<\overline{f}$. By the minimality of $\bar{f}$ in Eq.~(\mref{eq:F}), we see that $f-q|_s\in \bfk\,\Irr(S)+\Id(S)$ and hence also $f\in \bfk\,\Irr(S)+\Id(S)$, again a contradiction. Therefore, $\bfk\mapm{Z}=\bfk  \Irr(S)+\Id(S)$.

Suppose $\bfk\Irr(S)\cap \Id(S)\neq 0$ and let $0\neq f\in \bfk  \Irr(S)\cap \Id(S)$. Then $f=c_1v_1+\cdots+c_kv_k$ with
$v_1>\cdots >v_k\in  \Irr(S)$. Then by $f\in \Id(S)$ and
Part~(\mref{it:cd3}), $\bar{f}=\nu_1$ is of the form $
q|_{\bar{s}}$ for some $q\in \mapm{Z}^\star$ and $s\in S$. This
is a contradiction to the construction of $ \Irr(S)$.

Therefore, $\bfk\mapm{Z}=\bfk  \Irr(S)\oplus \Id(S)$ and hence $ \Irr(S)$ is a
basis of $\bfk\mapm{Z}/\Id(S)$.

\smallskip

\noindent (\mref{it:cd4}) $\Longrightarrow$ (\mref{it:cd1}).
We first prove a lemma.
\begin{lemma}
Suppose Item~(\mref{it:cd4}) holds. Let $0\neq h\in \Id(S)$ and let $w\in \mapm{Z}$ such that $w>\bar{h}$. Then $h=\sum_j d_j q_j|_{s_j}$ with $q_j|_{\overline{s_j}}<w$.
\mlabel{lem:idealw}
\end{lemma}

\begin{proof}
Denote
$$ \lead(S):=\left\{\,q|_{\bar{s}} \,\big|\, q\in \mapm{Z}^\star, s\in S\,\right\}.$$
Then by Item~(\mref{it:cd4}), we have the disjoint union $\mapm{Z}=\lead(S)\sqcup \Irr(S)$. Then for $0\neq h\in \Id(S)$, we can write
$$h=c_1u_1+\cdots +c_ku_k$$
in which $u_1>\cdots>u_k\in \mapm{Z}$ and there is $1\leq i_0\leq k$ such that $u_{i_0}\in \lead(S)$ and all the previous terms, if there are any, are in $\Irr(S)$.
We call $u_{i_0}$ the first monomial of $h$ in $\lead(S)$.
Suppose the conclusion of the lemma does not hold. Then we can choose our counter example $h$ such that the first monomial $u_{i_0}$ of $h$ is minimal with respect to the order $<$\,. Then we have $u_{i_0}=q|_{\bar{s}}$ for some $s\in S$. Consider
$$h':=h-q|_{s}=c_1u_1+\cdots + c_{i_0-1}u_{i_0-1} + c_{i_0}q_{\bar{s}-s} +c_{i_0+1}u_{i_0+1} +\cdots  +c_ku_k.$$
Then we still have $\overline{h'}<w$.
Since $h$ is a counter example, $h'\neq 0$. Since $q|_{s}$ is in $\Id(S)$, $h'$ is still in $\Id(S)$. Since $$\overline{q|_{\bar{s}-s}} =q|_{\overline{\bar{s}-s}}<q|_{\bar{s}} = u_{i_0},$$
the first monomial of $h'$ in $\lead(S)$ is smaller than $u_{i_0}$. By the minimality of $h$, we have $h'=\sum_j d_j q_j|_{s_j}$ with $q_j|{\overline{s_j}}<w$. Then $h=h'+q|_{s}$ also has this property. This is a contradiction.
\end{proof}

Now suppose
$f,g\in S$ give a composition. Let $F=f\mu$ and $G=\nu g$ in the
case of intersection composition and let $F=f$ and $G=q|_g$ in the
case of including composition. Then we have $w:=\bar{F}=\bar{G}$.
If $(f,g)_w=F-G=0$, then there is nothing to prove.
If $(f,g)_w\neq 0$, then by Lemma~\mref{lem:idealw}, there are $q_j\in \mapmXstar$ and $s_j\in S$ such
that $(f,g)_w=\sum_j d_j q_j|_{s_j}.$ with $q_j|_{\overline{s_j}}<w$. Hence $(f,g)_w$ is trivial modulo $(S,w)$.
\end{proof}

\section{Differential type operators}
\mlabel{sec:diff}
As remarked in the Introduction, we restrict our attention to those
OPIs that are computationally feasible, in particular, to two
families that are broad enough to include all the operators in
Rota's list, except the Reynolds operator. These families are
identified by how they behave with respect to multiplication for
which associativity is assumed. As differentiation is easier than
integration, we progress more on differential type OPIs than on
Rota-Baxter type ones.

\subsection{Concepts and conjecture}
Our model for differential type operators is the free differential
algebra and its weighted generalization as considered
in~\mcite{G-K3}. We refer the reader there for further details on
construction of free (noncommutative) differential algebras of
weight $\lambda$.

\subsubsection{The concepts}
The known OPIs that define  an endomorphism operator, a differential
operator, or a differential operator of weight $\lambda$ share a
common pattern, based on which we will define OPIs of differential
type. For this family of operators, we shall use the prefix notation
$\delta(r)$ (or $\delta r$) for the image of $r$ in such an algebra,
which is more traditional, but we shall continue to use the infix
notation $\lc r \rc$ in $\bfk\mapm{X}$ to emphasize the string nature of bracketed expressions.

\begin{mdefn} We say an expression
$E(X) \in \bfk\mapm{X}$ is {\bf in differentially reduced form}
\paren{DRF} if it does not contain any subexpression of the form $\lc uv
\rc$ for any non-units $u,v \in \bfk\mapm{X}$. Let $\Sigma$ be a
rewriting system {\em \cite{BN}} in $\bfk\mapm{X}$. We say $E(X)$ is
{\bf $\Sigma$-reducible} if $E(X)$ can be reduced to zero under
$\Sigma$. \mlabel{def:DRF}
\end{mdefn}

Let a set $X$ be given. Define $x^{(n)}\in \mapm{X}, n\geq 0,$
recursively by
$$ x^{(0)}=x, x^{(k+1)}=\lc x^{(k)}\rc, k\geq 0.$$
Then
\begin{equation}
\Delta(X):=\{x^{(n)}\,|x\in X, n\geq 0\},
\mlabel{eq:Delta}
\end{equation}
generates a monoid $M(\Delta(X))$ in
$\mapm{X}$ and hence $\bfk\langle\Delta(X)\rangle:=\bfk
M(\Delta(X))$ (the noncommutative differential polynomial ring) is a
subalgebra of $\bfk\mapm{X}$. Then $E(X)\in \bfk\mapm{X}$ is in DRF
if and only if it is in $\bfk\langle \Delta(X)\rangle$.

\begin{mdefn}  Let $\phi(x,y) := \lc xy\rc - N(x,y) \in \bfk\mapm{x,y}$.
\begin{enumerate}
\item
Define an associated rewriting system
\begin{equation}\Sigma_\phi := \left \{\lc a b\rc \mapsto N(a,b) \mid
a, b \in \mapmonoid(Z)\backslash\{1\} \right\},
\mlabel{eq:Sphi}
\end{equation}
where $Z$ is a set.
More precisely, for $g, g'\in \bfk\mapm{Z}$, denote
$g\to_{\Sigma_\phi}g'$ if there are $q\in \mapmonoid^\star(Z)$ and $a,b\in \mapmonoid(Z)$ such that
\begin{enumerate}
\item $q|_{\lc ab\rc}$ is a monomial of $g$ with coefficient $c\neq 0$,
\item $g'=g-cq|_{(\lc ab\rc-N(a,b))}$.
\end{enumerate}
In other words, $g'$ is obtained from $g$ by replacing a subword $\lc ab\rc$ in a monomial of $g$ by $N(a,b)$.
\item
An expression $E(a,b) \in \bfk\mapm{Z}$ is {\bf differentially
$\phi$-reducible} if it is
$\Sigma_\phi$-reducible.
\end{enumerate}
\mlabel{def:phireduced}
\end{mdefn}
The non-unit requirement in Eq.~(\mref{eq:Sphi}) is
to avoid infinite rewriting of the form such as $\lc u \rc =\lc u
\cdot \bfone \rc \mapsto N(u, \bfone)$, when $N(u,\bfone)$ may
involve $\lc u \rc$.
See Section~\mref{ss:diffGS} for this rewriting system in terms of reduction relations.

\begin{mdefn}
We say an OPI $\phi\in \bfk\mapm{x,y}$, or the expression $\phi = 0$, is {\bf of
differential type} \paren{OPIDT} if $\phi$ has the form $\lc xy\rc -
N(x,y)$, where $N(x,y)$ satisfies the three conditions$:$
\begin{enumerate}
\item
$N(x,y)$ is {\bf totally linear} in $x$ and $y$, in the sense that the total degree of $\lc x\rc^n, n\geq 0$ (resp. $\lc y\rc^n, n\geq 0$) in each monomial of $N(x,y)$ is one;
\mlabel{it:diff0}
\item
$N(x,y)$ is in DRF; \mlabel{it:diff1}
\item
For any set $Z$ and $u, v, w\in \mapmonoid(Z)\backslash\{1\}$, $ N(uv,w) - N(u,vw)$ is differentially $\phi$-reducible. \mlabel{it:diff2}
\end{enumerate}
If $\phi := \lc xy\rc - N(x,y)$ is an OPIDT, we also say the
expression $N(x,y)$ and the defining operator $P$ of a
$\phi$-algebra $R$ are {\bf of differential type}.
\mlabel{de:difftype1}
\end{mdefn}

\begin{mremark}{\em Condition \mref{it:diff0} is imposed since we are only interested in linear operators.
Condition \mref{it:diff1} is needed to avoid infinite
rewriting under $\Sigma_\phi$. Condition~\mref{it:diff2} is needed so that
$\lc(uv)w\rc=\lc u(vw)\rc$. Note that Condition \mref{it:diff2} is {\it
not} equivalent to
$$\phi_{\bfk\mapm{Z}}(uv,w) - \phi_{\bfk\mapm{Z}}(u,vw) \in
I_\phi(\{Z\})\ \forall\ u,v,w \in \bfk\mapm{Z},$$ which is always
true. Here $I_\phi(\{Z\})$ is the operated ideal of $\bfk\mapm{Z}$
generated by the set $$\left\{\phi_{\bfk\mapm{Z}}(a,b)\,|\, a, b\in
\bfk\mapm{Z}\right\}.$$ }\end{mremark}

\begin{exam}{\em For any $\lambda \in \bfk$, the
expressions $\lambda xy$, $\lambda \lc x\rc \lc y \rc$ (operators
that are semi-endomorphisms), and $\lambda \lc y \rc \lc x \rc$
(operators that are semi-anti\-morphisms) are of differential type. A
differential operator of weight $\lambda$ satisfies an OPI of
differential type (Eq.\,(\mref{eq:dlambda})). This can be easily
verified.} \mlabel{ex:opidt}\end{exam}

\subsubsection{The OPIDT conjecture}

We can now state the classification problem of differential type
OPIs and operators
\begin{problem}
{\bf (Rota's Problem: the Differential Case) } Find all operated
polynomial identities of differential type by finding all
expressions $N(x,y)\in \bfk\mapm{x,y}$ of differential type.
\end{problem}

We propose the following answer to this problem.

\begin{conjecture} {\bf (OPIs of Differential Type)}
Let $\bfk$ be a field of characteristic zero. Every expression
$N(x,y) \in \bfk\mapm{x,y}$ of differential type takes one \paren{or
more} of the forms below for some $a,b,c,e \in \bfk:$
\begin{enumerate}
\item $b(x\lc y\rc +\lc x\rc y)+c\lc x\rc
\lc y\rc +exy$ where $b^2 = b + ce$,\mlabel{dtoi:one}
\item $ce^2 yx + e xy + c \lc y\rc \lc x\rc  -ce(y\lc x\rc +
\lc y\rc x)$,\mlabel{dtoi:two}
\item
$\sum\limits_{i,j\geq 0} a_{ij}\, \lc 1\rc^i xy \lc 1\rc^j$ with the convention that $\lc 1\rc^0=1$. \item $x\lc y\rc +\lc x\rc y + a x\lc 1\rc
 y + b x y$,\mlabel{dtoi:four}
\item $\lc x\rc y+a(x\lc 1\rc y-xy\lc 1\rc )$,\mlabel{dtoi:five}

\item $x\lc y\rc +a(x\lc 1\rc y-\lc 1\rc xy)$.\mlabel{dtoi:six}
\end{enumerate} 
 \mlabel{con:diffclass}
\end{conjecture}

Note that the list is not symmetric in $x$ and $y$. One might think
that if $N(x,y)$ is of differential type, then so is $N(y,x)$. But
this is not true.

\begin{exam} $N_1(x,y):=x\lc y\rc$ is of differential type since
\begin{eqnarray*}N_1(uv,w) - N_1(u,vw) &=& uv \lc w\rc - u \lc vw\rc
\\&\mapsto& uv
\lc w\rc - uv \lc w\rc = 0\end{eqnarray*}
for all $u,v,w \in
\mapmonoid(Z)$. However, $N_2(x,y):=y\lc x\rc$ is not, since
\begin{eqnarray*}N_2(xy,x) - N_2(u,vu) &=& u \lc uv\rc - vu\lc
u\rc\\&\mapsto& u v \lc u \rc - vu\lc u\rc = (uv-vu)\lc u
\rc,\end{eqnarray*} which is in DRF \paren{no further reduction
using $\Sigma_\phi$ is possible, where $\phi:= \lc xy\rc -
N_2(x,y)$} but non-zero. See also {\em Example \mref{ex:freak}}.
\mlabel{ex:dt}\end{exam}

\subsection{Evidence for the conjecture}
\mlabel{ss:evid} We provide evidence, both computational and
theoretical, for Conjecture~\mref{con:diffclass}. Further results will be given in Section~\ref{ss:diffGS}.

\subsubsection{Verification of the operators}
\mlabel{sss:veri}
\begin{theorem}{\sl
The OPI $\phi := \lc xy\rc - N(x,y)$, where $N(x,y)$ is any
expression listed in Conjecture~\mref{con:diffclass} is of
differential type. \mlabel{thm:exam}}
\end{theorem}

\begin{proof} Clearly, all six expressions are in DRF. We
check $\phi$-reducibility for the first two cases.

\smallskip
\noindent {\bf Case 1}.  Here $N(x,y):=b(x\lc y\rc +\lc x\rc y)+c\lc
x\rc \lc y\rc +exy$, where $b^2 = b + ce$. We have
\begin{equation}
cN(x,y)+bxy=(c\lc x\rc + bx)(c\lc y\rc +by). \label{eq:phi}
\end{equation}
Let $\alpha$ be the operator defined by $\alpha(u):=c\lc u\rc +bu$
for $u \in \bfk\mapm{x,y}$. Then for any non-units $u,v \in
\bfk\mapm{x,y}$, the rewriting rule $\lc uv\rc\mapsto N(u,v)$ gives
the rewriting rules
$$ \alpha(uv)=c\lc uv \rc+buv \mapsto cN(u,v)+buv=\alpha(u)\alpha(v)$$
by Eq.~(\mref{eq:phi}). Again, by Eq.~(\mref{eq:phi}), for a
non-unit $w$, we have
\begin{eqnarray*}cN(uv,w)+b(uv)w &=&
\alpha(uv)\alpha(w)\mapsto (\alpha(u)\alpha(v))\alpha(w),\\
cN(u,vw)+bu(vw)&=&\alpha(u)\alpha(vw) \mapsto
\alpha(u)(\alpha(v)\alpha(w)).\end{eqnarray*} Then $c(N(uv,w) -
N(u,vw))$ is differentially $\phi$-reducible by associativity.  If
$c \ne0$, then $N(x,y)$ is of differential type. Suppose $c = 0$.
The constraint $b^2=b+ce$ becomes $b^2=b$ and either $b=0$ or $b=1$.
When $b=0$, $\phi = \lc xy \rc - e xy$ (semi-endomorphism case), and
when $b=1$, $\phi = \lc xy\rc - (x\lc y\rc + \lc x\rc y +exy)$.
These are easily verified directly to be OPIs of differential type.

\smallskip
\noindent {\bf Case 2}. Here $N(x,y):=ce^2 yx + e xy + c \lc y\rc
\lc x\rc -ce(y\lc x\rc +\lc y\rc x)$ and we have $N(x,y)-exy=c(\lc
y\rc -ey)(\lc x\rc -ex)$.  Let $\alpha(u)=\lc u\rc -eu$ and the rest
of the proof is similar to Case 1.

\smallskip For the remaining cases, it is routine to check that
$N(uv,w)-N(u,vw)$ is differentially $\phi$-reducible for $\phi :=
\lc xy\rc - N(x,y)$. For example, for Case 5, we have, using
associativity,
\begin{eqnarray*}N(uv,w)&=&
\lc uv\rc w +a(uv\lc 1\rc w - uvw\lc 1\rc)\\
&\mapsto& \bigl(\lc u\rc v +a (u\lc \bfone\rc v - uv \lc \bfone
\rc)\bigr)w +a(uv\lc 1\rc w - uvw\lc 1\rc)\\
&=& \lc u\rc
vw+a(u\lc 1\rc vw-uvw \lc 1\rc)\\
&=& N(u, vw).  \end{eqnarray*}
\end{proof}

\subsubsection{Computational evidence}
\mlabel{sss:com} \begin{mdefn}The {\bf operator degree} of a
monomial in $\bfk\mapm{\opset}$ is the total number that the
operator $\lc\ \rc$ appears in the monomial. The {\bf operator
degree} of a polynomial $\phi$ in $\bfk\mapm{\opset}$ is  the
maximum of the operator degrees of the monomials appearing in
$\phi$.\end{mdefn}

\begin{theorem}{\sl
Let $\bfk$ be a field. The only expressions $N(x,y)$ of differential
type for which the total operator degrees $\leq 2$ are the ones
listed in Conjecture~\mref{con:diffclass}. More precisely, the only
expressions of differential type in the form
\begin{eqnarray*}
\lefteqn{N(x,y):=a_{0,0}xy+a_{0,1}x\lc y \rc + a_{0,2}x \lc \lc y \rc \rc +
a_{1,0}\lc x \rc y + a_{1,1}\lc x \rc \lc y \rc}\\&& +a_{1,2} \lc x
\rc \lc \lc y \rc \rc + a_{2,0} \lc \lc x \rc \rc y + a_{2,1} \lc
\lc x \rc \rc \lc y \rc + a_{2,2} \lc \lc x \rc \rc \lc \lc y \rc
\rc
\\&&b_{0,0}yx+b_{0,1} y \lc x \rc + b_{0,2} y \lc \lc x \rc \rc + b_{1,0} \lc
y \rc x + b_{1,1}\lc y \rc \lc x \rc\\&& + b_{1,2} \lc y \rc \lc \lc
x \rc \rc + b_{2,0} \lc\lc y\rc\rc x + b_{2,1}\lc \lc y \rc \rc \lc x \rc + b_{2,2} \lc \lc y
\rc \rc \lc \lc x \rc \rc
\end{eqnarray*} where $a_{i,j}, b_{i,j} \in \bfk$ $(0 \leq i, j\leq
2)$, are the ones listed. \mlabel{thm:lowdeg}}
\end{theorem}

\begin{proof} This is obtained and verified by computations in {\it
Mathematica} \mcite{Wolfram}. See Section \mref{sec:comp} for a
brief description and \mcite{Sit} for details and
results.\end{proof}

\section{Relationship of differential type operators with convergent rewriting and Gr\"obner-Shirshov bases}
\label{ss:diffGS} We now characterize OPIDT in terms of convergent rewriting systems and
Gr\"obner-Shirshov bases as we have discussed in Section~\mref{sec:GS}.
We quote the following basic result of well order for reference.
\begin{lemma}{\sl
\begin{enumerate}
\item Let $A$ and $B$ be two sets with well-orderings. Then we obtain an extended well order on the disjoint union $A\sqcup B$ by defining $a<b$ for all $a\in A$ and $b\in B$.
    \label{it:union}
\item Let $A$ be a set with a well order. Then the lexicographic order on $M(A)$ is a well order.
    \label{it:monoid}
\end{enumerate}
\label{lem:ord}}
\end{lemma}

Let $>$ be a well-ordering on a set $Z$. We extend $>$ to a well-ordering on $\mapmonoid(Z)=\dirlim \mapmonoid_n(Z)$ by recursively defining a well-ordering $>_n$, on $\mapmonoid_n:=\mapmonoid_n(Z)$ for each $n\geq 0$. Denote by $\deg_{_{Z}}(u)$ the number of $x\in Z$ in $u$ with repetition.
When $n=0$, we have $\mapmonoid_0=M(Z)$. In this case, we obtain a well-ordering by taking the lexicographic order $>_{\text{lex}}$ on $M(Z)$ induced by $>$ with the convention that $u>_{\text{lex}} 1$ for all $u\in M(Z)\backslash\{1\}$.
Suppose $>_n$ has been defined on $\mapmonoid_n:=M(Z\sqcup \lc \mapmonoid_{n-1}\rc)$ for an $n\geq 0$. Then $>_n$ induces
\begin{enumerate}
\item
a well-ordering $>'_n$ on $\lc\mapmonoid_n\rc$ by
\begin{equation}
\lc u\rc >'_n \lc v\rc \Longleftrightarrow u>_n v;
\mlabel{eq:orderbr}
\end{equation}
\item
then a well-ordering $>''_n$ on $Z\sqcup \lc \mapmonoid_n\rc$ by Lemma~\mref{lem:ord}.(\mref{it:union}); \item
then a well-ordering $>'''_n$ on $Z\sqcup \lc \mapmonoid_n\rc$ by
\begin{equation}
u>'''_n v\Longleftrightarrow \left\{ \begin{array}{l} \text{\ either\ } \deg_{_Z}(u)>\deg_{_{Z}} (v) \\
\text{or\ } \deg_{_Z}(u)=\deg_{_Z}(v) \text{\ and\ } u>''_n v. \end{array} \right .
\mlabel{eq:diffordn}
\end{equation}
\item
then the lexicographic well-ordering $>_{n+1}$ on $\mapmonoid_{n+1}=M(Z\sqcup \lc \mapmonoid_n\rc)$ induced by $>_n'''$.
\end{enumerate}
The orders $>_n$ are compatible with the direct system $\{\mapmonoid_n\}_{n\geq 0}$ and hence induces a well-ordering, still denoted by $>$, on
$\mapmonoid(Z)=\dirlim \mapmonoid_n$.

\begin{exam}
Under this order, $\lc xy\rc$ is greater than $1, x, y$ and their iterated operations under $\lc\ \rc$. Thus $\lc xy\rc$ is the leading term for $\phi(x,y)=\lc xy\rc -N(x,y)$ when $N(x,y)$ is in DRF, in particular, for those $N(x,y)$ listed in Conjecture~\mref{con:diffclass}.
\end{exam}

\begin{lemma}{\sl
The order $>$ on $\mapmonoid(Z)$ is a monomial order.
\label{lem:mord}}
\end{lemma}
\begin{proof}
We prove by induction on $n\geq 0$ the claim that for any $q\in \mapmonoid^\star(Z)\cap \mapmonoid_n(Z\sqcup \{\star\})$, $u>v$ in $\mapmonoid(Z)$ implies $q|_u>q|_v$.

When $n=0$, we have $q\in M(Z\sqcup \{\star\})$ in which $\star$ only appears once. Thus $q=a\star b$ with $a,b\in M(Z)$. Thus $u>v$ in $\mapmonoid(Z)$ implies that $aub>avb$ by the definition of lexicographic order.

Suppose the claim has been proved for all $q\in
\mapmonoid^\star(Z)\cap \mapmonoid_n(Z\sqcup \{\star\})$ for an
$n\geq 0$. Consider $q\in \mapmonoid^\star(Z)\cap
\mapmonoid_{n+1}(Z\sqcup \{\star\})$. Then $q=apb$ with $p\in
\mapmonoid^\star(Z)\cap \mapmonoid_{n+1}(Z\sqcup \{\star\})$ being
indecomposable and $a, b\in \mapmonoid_{n+1}(Z)$. Thus $p\in Z$ is
impossible. So we have $p\in \lc \mapmonoid_n(Z\sqcup \{\star\})\rc$.
Then $p=\lc p'\rc$ and $p$ is in $\mapmonoid^\star(Z)\cap\mapmonoid_n(Z\sqcup
\{\star\})$. Thus by the induction hypothesis, if $u>v$, then
$p'|_u>p'|_v$. Then by Eq.~(\mref{eq:orderbr}), we also have $p|_u>p|_v$ and hence
$q|_u>q|_v$ by the lexicographic order. This completes the
induction.
\end{proof}

We next extend the concept of reduction relation from polynomial algebras $\bfk[Z]$~\cite[Section 8.2]{BN} to operated polynomial algebras $\bfk\mapm{Z}$.

\begin{mdefn}
Let $Z$ be a set and let $<$ be a monomial well-ordering on $\mapmonoid(Z)$. Let $f\in \bfk\mapm{Z}$ be monic. We use $f$ to define the following {\bf reduction relation} $\rightarrow_f$: For $g, g'\in \bfk\mapm{Z}$, define $g\rightarrow_f g'$ if there is $q\in \mapmonoid^\star(Z)$ such that
\begin{enumerate}
\item
$q|_{\overline{f}}$ is a monomial of $g$ with coefficient $c$,
\item
$g'=g-cq|_{f}$.
\end{enumerate}
\mlabel{de:redrel}
\end{mdefn}
In other words, $g'$ is obtained by replacing a subword $\overline{f}$ in a monomial of $g$ by $\overline{f}-f$.
If $F$ is a set of monic bracketed polynomials, we define
$$\rightarrow_F:= \cup_{f\in F} \rightarrow_f.$$

We refer the reader to~\cite{BN} for concepts in rewriting systems, such as joinable and convergence.

\begin{mprop}
{\sl Let $Z$ be a set and let $\mapmonoid(Z)$ be equipped with a
monomial well-ordering $<$. Let $F$ be a set of monic bracketed
polynomials. Then the reduction relation $\rightarrow_F$ is a
terminating relation. \mlabel{pp:term}}
\end{mprop}
See~\cite[Prop. 8.2.9]{BN} for the case of polynomials.

\begin{proof}
For each $f\in \bfk\mapm{Z}$, let $M(f)$ denote the set of monomials in $f$. Let $>_{mul}$ denote the multiset order on the set $\calm(\mapmonoid(Z))$ of finite multisets over $\mapmonoid(Z)$ induced by $>$ on $\mapmonoid(Z)$. Then by \cite[Theorem~2.5.5]{BN}, the order $>_{mul}$ is terminating. Thus we just need to show that if $g\to_F g'$, then $M(g)>_{mul} M(g')$.
If $g\to_F g'$, then there are $f\in F$, $q\in \mapmonoid^\star(Z)$ such that $q|_{\overline{f}}$ is a monomial of $g$ with coefficient $c\neq 0$ and such that $g'=g-cq|_f$. Since $<$ is a monomial well order, all terms in $q_{f-\overline{f}}$ are smaller than $q|_{\overline{f}}$. Thus $M(g')$ is obtained from $M(g)$ by replacing the monomial $q|_{\overline{f}}$ by smaller monomials. This implies $M(g)>_{mul} M(g')$.
\end{proof}

We also prove the following variation of \cite[Lemma 8.3.3]{BN}.
\begin{lemma}
Let $f, g\in \bfk\mapm{Z}$. If $f-g$ is reduced to zero. Then $f$ and $g$ are joinable.
\mlabel{lem:833}
\end{lemma}

\begin{proof}
We use induction on the number $n$ of iterations of applying $\to_F$ to $f-g$ to get zero. If $n=0$, then $f-g=0$ and there is nothing to prove. Suppose the conclusion of the lemma holds with $n\geq 0$ iterations and consider the case of $n+1$. Suppose the first reduction relation is $\to_{f_i}$ for an $f_i\in F$ by applying $f_i$ to a monomial $m$ and $m$ appears in $f$ (resp. $g$) with coefficient $a$ (resp. $b$). So $m=q|_{\overline{f_i}}$ for some $q\in \mapmonoid^\star(Z)$. Then we obtain $f-g \to_{f_i} h$ where
$$ h= (f-g)-(a-b)q|_{f_i} = (f-aq|_{f_i})-(g-bq|_{f_i}).$$
Since $h$, that is the right hand side, is reduced to zero with $n$ iterations of reductions, by the induction hypothesis, $f-aq|_{f_i}$ and $g-bq|_{f_i}$ are joinable. Then it follows that $f$ and $g$ are joinable. \end{proof}

\begin{theorem}{\sl
Let $\phi(x,y):=\delta(xy)-N(x,y)\in \bfk\mapm{x,y}$ with $N(x,y)$
in DRF and totally linear in $x, y$. The following statements are equivalent.
\begin{enumerate}
\item
$\phi(x,y)$ is of differential type;
\mlabel{it:maind}
\item
The rewriting system $\Sigma_\phi$ is convergent;
\mlabel{it:mainc}
\item
Let $Z$ be a set with a well-ordering. With the order $>$ in Eq.~(\ref{eq:diffordn}), the
set
$$
S:=S_\phi:=\left\{\phi(u,v)=\delta(uv)-N(u,v) |\  u,v \in
\mapmonoid(Z)\backslash\{1\}\right\}
$$
is a Gr\"{o}bner-Shirshov basis in $\bfk\mapm{Z}$.
\mlabel{it:maings}
\item
The free $\phi$-algebra on a set $Z$ is the noncommutative polynomial $\bfk$-algebra $\bfk\langle \Delta(Z)\rangle$ where $\Delta(Z)$ is defined in Eq.~(\mref{eq:Delta}), together with the operator $d:=d_Z$ on $\bfk\langle \Delta(Z)\rangle$ defined by the following recursion:

Let $u=u_1u_2\cdots u_k \in
M(\Delta(Z))$, where $u_i\in \Delta(Z), 1\leq i\leq k$.
\begin{enumerate}
\item If $k=1$, i.e., $u=\delta^i(x)$ for some $i\geq 0, x\in Z $,
 then define $d(u)=\delta^{i+1}(x)$.
\item
If $k\geq 1$, then recursively define
$
d(u)=N(u_1,u_2\cdots u_k).
$
\end{enumerate}
\mlabel{it:mainf}
\end{enumerate}
\mlabel{thm:gsdiff}}
\end{theorem}

By Theorem~\mref{thm:exam}, we have
\begin{coro}
Let $N(x,y)$ be from the list in Conjecture~\mref{con:diffclass}. Then all the statements in Theorem~\mref{thm:gsdiff} hold.
\mlabel{co:gsdiff}
\end{coro}
When $N(x,y)=x\delta(y)+\delta(x)y+\lambda \delta(x)\delta(y),$ we
obtain \cite[Theorem~5.1]{BCQ}.

\begin{proof}
(\mref{it:maind}) $\Longrightarrow$ (\mref{it:mainc}) We first note that the rewriting system $\Sigma_\phi$ in Definition~\mref{def:phireduced} is the same as the reduction relation $\to_{S_\phi}$ with
$$S_\phi:=\{\phi(u,v)\,| u,v\in \bfk\mapm{Z} \},$$
with the order in Eq.~(\mref{eq:diffordn}). Thus by Proposition~\mref{pp:term}, $\Sigma_\phi$ is terminating.
Consequently, by \cite[Lemma 2.7.2]{BN}, to prove that $\Sigma_\phi$ is confluent and hence convergent, we just need to prove that $\Sigma_\phi$ is locally confluent. Suppose $g_1{\ }_{\Sigma_\phi}\!\!\!\!\leftarrow f \rightarrow_{\Sigma_\phi} g_2$ for $f\in \mapmonoid(Z)$ and $g_1, g_2\in \bfk\mapm{Z}.$ Then there are $q_1, q_2\in \mapmonoid^\star(Z)$ and $s_1, s_2\in S_\phi(Z)$ such that $$q_1|_{\overline{s_1}}=f=q_2|_{\overline{s_2}}, \ g_1=q_1|_{\overline{s_1}-s_1},\ g_2=q_2|_{\overline{s_2}-s_2}.$$
Since $s_1, s_2$ are in $S_\phi(Z)$, we can write
\begin{eqnarray}
 s_1 &=& \phi(u,v)=\delta(uv)-N(u,v)=\delta(uv)-\sum_{i} c_i \phi_i(u,v),\nonumber\\
s_2 &=& \phi(r,s)=\delta(rs)-N(r,s)=\delta(rs)-\sum_{i} c_i \phi_i(r,s),
\mlabel{eq:phimon}
\end{eqnarray}
for some $u, v, r, s \in \mapmonoid(Z)\backslash\{1\}$. Here we have used the notation
$$N(x,y)=\sum_{i=1}^k c_i \phi_i(x,y), \phi_i(x,y)\in \mapmonoid(x,y), 1\leq i\leq k.$$
As in the proof of Lemma~\mref{lem:Scong}, there are three cases to consider.
\smallskip

\noindent {\bf Case I.} Suppose the bracketed words $\overline{s_1}$
and $\overline{s_2}$ are disjoint in $f$. Let $q\in
\mapmonoid^{\star_1,\star_2}(X)$ be the
$(\star_1,\star_2)$-bracketed word obtained by replacing this
occurrence of $\overline{s_1}$ (resp. $\overline{s_2}$) in $f$ by
$\star_1$ (resp. $\star_2$). Then we have
$$f=q|_{\overline{s_1},\
\overline{s_2}}=q_1|_{\overline{s_1}}=q_2|_{\overline{s_2}}.
$$
Then we have
$$ q_1|_{s_1}=q|_{s_1,\overline{s_2}}, \quad q_2|_{s_2}=q|_{\overline{s_1},s_2}.$$
Hence
$$ g_1=q_1|_{s_1-\overline{s_1}} =q|_{s_1-\overline{s_1},\overline{s_2}}
\mapsto_{\Sigma_\phi} q|_{s_1-\overline{s_1}, s_2-\overline{s_2}}.$$
Similarly, $g_2 \mapsto_{\Sigma_\phi} q|_{s_1-\overline{s_1}, s_2-\overline{s_2}}.$
This proves the local confluence.

\smallskip

\noindent {\bf Case II.} Suppose the bracketed words
$\overline{s_1}$ and $\overline{s_2}$ have nonempty intersection in
$f$ but are not a proper subword of each other. Since $\overline{s_1}=\delta(uv)$ and $\overline{s_2}=\delta(rs)$ are indecomposable in $\mapmonoid(Z)$, this is possible only when $\delta(uv)=\delta(rs)$. Thus $uv=rs$. Factoring each of $u,v,r,s$
into standard decompositions, we see that there are $a,b,c\in
\mapmonoid(Z)$ such that $u=ab, v=c$ and $r=a, s=bc$.
Then we have
$\bar{s_1}=\delta(abc)=\bar{s_2}$ and

$$
g_1-g_2 =N(ab,c)-N(a,bc). $$

Since $u,v,r,s\neq 1$, We have $a,c\neq 1$. If $b=1$, then $g_1-g_2$ is already zero. If $b\neq 1$, then since $\phi$ is of differential type, $g_1-g_2$ is reduced to zero. Then by Lemma~\mref{lem:833}, $g_1$ and $g_2$ are joinable.
\smallskip

\noindent {\bf Case III.} Suppose one of the bracketed words
$\overline{s_1}$ and  $\overline{s_2}$ is contained in the other.
Without loss of generality, suppose
$\overline{s_1}=q|_{\overline{s_2}}$ for some $\star$-bracketed word $q\in \mapmonoid^\star(Z)$. This means that
$\delta(uv)=\overline{s_1}=q|_{\overline{s_2}}=q|_{\delta(rs)}.$
Then $q=\delta(q')$ for some $\star$-bracketed word $q'$ and hence $\delta(uv)=q|_{\delta(rs)}=\delta(q'|_{\delta(rs)}).$
This gives $uv=q'|_{\delta(rs)}$. Since $u,v\in \mapmonoid(Z)\backslash\{1\}$, we have either $q'=pv$ with $p|_{\delta(rs)}=u$ or $q'=up$ with $p|_{\delta(rs)}=v$, where $p\in \mapmonoid^\star(Z)$.
Without loss of generality, suppose $q'=pv$ with $p|_{\delta(rs)}=u$.
Then we have
$$N(p|_{\delta(rs)},v)=N(u,v){\ }_{\Sigma_\phi}\!\!\!\leftarrow \delta(uv) = \delta(p|_{\delta(rs)} v) \rightarrow_{\Sigma_\phi} \delta(p|_{N(r,s)}v).$$
Using the notations in Eq.~(\mref{eq:phimon}), we obtain \begin{eqnarray*}
N(p|_{\delta(rs)},v)-\delta(p|_{N(r,s)}v)
&=& \sum_{i=1}^k c_i \phi_i(p|_{\delta(rs)},v) -\sum_{i=1}^k c_i
\delta(p|_{\phi_i(r,s)} v)\\
&\mapsto_{\Sigma_\phi}&
\sum_{i=1}^k c_i \phi_i(p|_{N(r,s)},v) -\sum_{i=1}^k c_i
N(p|_{\phi_i(r,s)}, v)
\\&=&
\sum_{i=1}^k c_i \sum_{j=1}^k c_j\phi_i(p|_{\phi_j(r,s)},v) -\sum_{i=1}^k c_i \sum_{j=1}^k c_j
\phi_j(p|_{\phi_i(r,s)}, v)
\\&=& 0
\end{eqnarray*}
since the two double sums become the same after exchanging $i$ and $j$.

Note that $N(p|_{\delta(rs)},v)=\overline{s_1}-s_1$ and $\delta(p|_{N(r,s)}v)=q|_{\overline{s_2}-s_2}$. We then see that $s_1-\overline{s_1}$ and $q|_{s_2-\overline{s_2}}$ are joinable by~\cite[Lemma 8.3.3]{BN}. Then $g_1=q_1|_{s_1-\overline{s_1}}$ and
$g_2=q_1|_{q|_{s_2-\overline{s_2}}}=(q_1|_{q})|_{s_2-\overline{s_2}} =q_2|_{s_2-\overline{s_2}}$ are joinable. This proves the local confluence in Case III and hence the proof of (\mref{it:maind}) $\Longrightarrow$ (\mref{it:mainc}).

\smallskip

\noindent
(\mref{it:mainc}) $\Longrightarrow$ (\mref{it:maings})
Suppose $\Sigma_\phi$ is convergent. We prove that all compositions from $S$ are trivial modulo
$(S,w)$.

\noindent {\bf (The case of intersection compositions). } By the
definition of $N(x,y)$ being in DRF, we have
$$\overline{\phi(x,y)}=\delta(xy).$$

Let two elements of $S$ be given. They are of the form
$$f:=\phi(u,v), \quad g:=\phi(r,s), \quad u,v,r,s\in \mapmonoid(Z)
\backslash\{1\}.$$ Hence $\bar{f}=\delta(uv)$ and
$\bar{g}=\delta(rs)$. Suppose $w=\bar{f}\mu=\nu\bar{g}$ gives an
intersection composition, where $\mu,\nu\in \mapm{X}$. Since
$(\bre{\bar{f}})=(\bre{\bar{g}})=1$, we must have
$\bre{w}<\bre{\bar{f}}+\bre{\bar{g}}=2$. Thus $\bre{w}=1$. This
means that $\bre{\mu}=\bre{\nu}=0$. Since $f, g$ are monic, we have
$\mu=\nu=1$. Thus $w=\bar{f}=\bar{g}$. That is,
$\delta(uv)=\delta(rs)$. Thus $uv=rs$. Factoring each of $u,v,r,s$
into standard decompositions, we see that there are $a,b,c\in
\mapm{X}$ such that $u=ab, v=c$ and $r=a, s=bc$. Therefore,
$f=\phi(ab,c)$ and $g=\phi(a,bc)$ is the only pair that gives
intersection composition. Then we have
$w=\bar{f}=\delta(abc)=\bar{g}$ and the resulting composition is

\begin{equation}
(f,g)_w:= f-g =-N(ab,c)+N(a,bc). \mlabel{eq:diffint}
\end{equation}

Since $N(ab,c){\ }_{\Sigma_\phi}\!\!\!\leftarrow \delta(abc) \rightarrow_{\Sigma_\phi} N(a,bc)$ and $\Sigma_\phi$ is confluent, we find that $N(ab,c)$ and $N(a,bc)$ are joinable. Hence $N(ab,c)-N(a,bc)$ is reduced to zero. In particular, $N(ab,c)-N(a,bc)$ is in $\Id(S)$. Since $\overline{\phi(ab,c)}=\delta(abc)=\overline{\phi(a,bc)}$, we have $\overline{N(ab,c)}<\delta(abc)$ and $\overline{N(a,bc)}<\delta(abc)$. Thus $N(ab,c)-N(a,bc)$ is trivial modulo $(S,\delta(abc))$.

\smallskip

\noindent {\bf (The cases of including compositions).} On the other
hand, $f$ and $g$ could only have the following including
compositions:
\begin{enumerate}
\item
If $u=p|_{\delta(rs)}$ for some $p\in \mapmonoid^{\star}(Z)$, then
$$w:=\overline{f}=q|_{\overline{g}}=\delta(p|_{\delta(rs)}v),$$
with $q:=\delta(pv)$.
\item
If $v=p|_{\delta(rs)}$ for some $p\in \mapmonoid^{\star}(Z)$, then
$$w:=\overline{f}=q|_{\overline{g}}=\delta(u\,p|_{\delta(rs)}),$$
with $q:=\delta(up)$.
\end{enumerate}

So we just need to check that in both cases these compositions are
trivial modulo $(S,w)$. Consider the first case. Using the
notation in Eq.~(\mref{eq:phimon}), this composition is

\begin{eqnarray*}
(f,g)_w&:=& f- q|_g \\
&=& \delta(uv)-\sum_{i=1}^kc_i\phi_i(u,v)
- \delta(p|_{g}\,v)\\
&=& \delta(p|_{\delta(rs)}v) -\sum_{i=1}^k c_i \phi_i(p|_{\delta(rs)},v)
-\left( \delta(p|_{\delta(rs)} v) -\sum_{i=1}^k c_i
\delta(p|_{\phi_i(r,s)} v)\right)\\
&=& -\sum_{i=1}^k c_i \phi_i(p|_{\delta(rs)},v) +\sum_{i=1}^k c_i
\delta(p|_{\phi_i(r,s)} v)\\
&=& -\sum_{i=1}^k c_i \phi_i(p|_{\phi{(r,s)}},v)-\sum_{i=1}^k c_i \phi_i(p|_{N(r,s)},v)+\sum_{i=1}^k c_i
\phi(p|_{\phi_i(r,s)} v)+\sum_{i=1}^k c_i
N(p|_{\phi_i(r,s)}, v)\\
&=&
-\sum_{i=1}^k c_i \phi_i(p|_{\phi{(r,s)}},v)+\sum_{i=1}^k c_i
\phi(p|_{\phi_i(r,s)} v)\\
&&
-\sum_{i=1}^k c_i \sum_{j=1}^k c_j \phi_i(p|_{\phi_j(r,s)},v)
+ \sum_{i=1}^k c_i \sum_{j=1}^k c_j \phi_j(p|_{\phi_i(r,s)}, v) \\
&=& -\sum_{i=1}^k c_i \phi_i(p|_{\phi{(r,s)}},v)+\sum_{i=1}^k c_i
\phi(p|_{\phi_i(r,s)}, v),
\end{eqnarray*}
since the double sums become the same after exchanging $i$ and $j$. Since $\overline{\phi(r,s)}=\delta(rs)$ we have $\phi_i(p|_{\overline{\phi(r,s)}}\,,
v)= \phi_i(p|_{\delta(rs)}, v)<w$. Thus the first sum is trivial
modulo $(S,w)$. Further every term $u_i:=\phi(p|_{\phi_i(r,s)},
v)$ in the second sum is already in $S$. So it is just
$\star|_{u_i}$ for the $\star$-bracketed word $\star$. We have
$$\overline{u_i}=\overline{\phi(p|_{\phi_i(r,s)}, v)}=
\delta(p|_{\phi_i(r,s)}\, v)<w.$$
Thus the second sum is also trivial modulo $(S,w)$. This proves
$(f,g)_w\equiv 0 \mod (S,w)$.

The proof of the second case is the same.

\smallskip

\noindent
(\mref{it:maings}) $\Longrightarrow$ (\mref{it:maind}) Suppose that a $\phi(x,y):=\delta(xy)-N(x,y)\in \bfk\mapm{x,y}$ with
$N(x,y)$ in DRF is such that
$$ S:=\{\phi(u,v)\,|\, u,v\in \bfk\mapm{Z}\}$$
is a Gr\"obner-Shirshov basis in $\bfk\mapm{Z}$ for any $Z$ with the
order $>$ in Eq.~(\mref{eq:diffordn}). Let $a,b,c\in
\mapmonoid(Z)\backslash \{1\}$. For $f=\phi(ab,c), g=\phi(a,bc)$, we
have $w:=\bar{f}\mu=\delta(abc)=\nu\bar{g}$ with $\mu=\nu=1$. Thus
we have an intersection composition
$$(f,g)_w^{1,1}:=f-g = -N(ab,c)+N(a,bc).$$
If $N(ab,c)=N(a,bc)$, then there is nothing to prove. If
$N(ab,c)-N(a,bc)\neq 0,$ then since $-N(ab,c)+N(a,bc)$ is in
$\Id(S)$ and $S$ is a Gr\"obner-Shirshov basis, by
Theorem~\mref{thm:CDL}, we have
$$ -N(ab,c)+N(a,bc)=\sum_{i=1}^n a_i q_i|_{s_i},$$
where $a_i\in \bfk, q_i\in \mapmonoid^\star(Z)$ and $s_i\in S, 1\leq
i\leq n$. This means that $-N(ab,c)+N(a,bc)$ is reduced to zero by
the rewriting system $\Sigma_\phi$ defined in Eq.~(\mref{eq:Sphi}).
Hence $\phi$ is of differential type.
\smallskip

\noindent
(\mref{it:mainf}) $\Longrightarrow$ (\mref{it:maings}) Suppose Item~\mref{it:mainf} holds. Then in particular $M(\Delta(Z))$ is a linear basis o $\bfk\mapm{Z}/I_{\phi}(Z)$. Then the conclusion follows from ($(4)\Longrightarrow (1)$) in Theorem~\mref{thm:CDL}.

\smallskip

\noindent
(\mref{it:maings}) $\Longrightarrow$ (\mref{it:mainf})
By Theorem~\mref{thm:CDL} and Corollary~\mref{co:GS}, $M(\Delta(Z))$ is a basis of the free $\phi$-algebra $\bfk\mapm{Z}/I_\phi(Z)$ in Proposition~\mref{pp:frpio}. Therefore, the restriction map
$$ \bfk\langle\Delta(Z)\rangle=\bfk M(\Delta(Z)) \to \bfk \mapm{Z} \to \bfk \mapm{Z}/I_\phi(Z)$$
is a linear isomorphism. Since $\bfk M(\Delta(Z))$ is closed under the multiplication on $\bfk\mapm{Z}$, we see that this linear isomorphism is an algebra isomorphism.
The recursive definition of the operator $d$ follows from the fact that it is the operator $\delta$ on $\bfk\mapm{Z}$ modulo $I_\phi(Z)$ and hence satisfies
$$\delta(uv)=N(u,v), \forall u, v\in M(\Delta(Z)).$$
\end{proof}

\section{Rota-Baxter type operators}
\mlabel{sec:RB}

We just give a brief discussion of Rota-Baxter type operators. Their study is more involved than differential type operators and will be left to a future work.

\begin{mdefn} We say an expression
$E(X) \in \bfk\mapm{X}$ is {\bf in Rota-Baxter reduced form}
\paren{RBRF} if it does not contain any subexpression of the form
$\lc u\rc \lc v \rc$ for any $u,v \in \bfk\mapm{X}$.
\mlabel{def:RBRF}
\end{mdefn}

\begin{mdefn} An OPI $\phi \in
\bfk\mapm{x,y}$ is {\bf of Rota-Baxter type} if it has the form $\lc
x\rc\lc y\rc-\lc M(x,y)\rc$ for some $M(x,y)\in \bfk\mapm{x,y}$ that
satisfies the two conditions:
\begin{enumerate}
\item
$M(x,y)$ is totally linear in $x, y$ in the sense that $x$ (resp. $y$) appears exactly once in each monomial of $M(x,y)$;
\item
$M(x,y)$ is in RBRF;
\item
$M(M(u,v),w)-M(u,M(v,w))$ is $\Pi_\phi$-reducible for all $u, v, w
\in \bfk\mapm{x,y}$, where $\Pi_\phi$ is the rewriting system
$$\Pi_\phi := \bigl\{ \lc a\rc \lc b\rc \mapsto \lc M(a,b)\rc
\ \mid\  a,b \in \bfk\mapm{x,y} \bigr\}.$$ \mlabel{it:rbii}
\end{enumerate}\vspace{-15pt}
If $\phi:=\lc x\rc\lc y\rc-\lc M(x,y)\rc$ is of Rota-Baxter type, we
also say the expression $M(x,y)$, and the defining operator $P$ of a
$\phi$-algebra $S$ are {\bf of Rota-Baxter type}. \mlabel{de:rbtype}
\end{mdefn}

\begin{exam} The expression
$M(x,y):=x\lc y\rc$ that defines the average operator is of
Rota-Baxter type since
\begin{eqnarray*}
M(M(u,v),w)-M(u,M(v,w))&=&M(u,v)\lc w\rc -u\lc M(v,w)\rc\\
&=&u\lc v\rc \lc w\rc -u\lc v \lc w\rc\rc\\
&\mapsto&u\lc v\lc w\rc\rc - u\lc v \lc w\rc\rc = 0.
\end{eqnarray*}
Other examples are OPIs corresponding to a Rota-Baxter operator or a
Nijenhuis operator.
\end{exam}

\begin{problem}{\bf (Rota's Classification Problem: Rota-Baxter Case)}
Find all Rota-Baxter type operators.
In other words, find all Rota-Baxter type expressions $M(x,y)\in
\bfk\mapm{x,y}$. \mlabel{pr:rbt}
\end{problem}

We propose the following answer to this problem.
\begin{conjecture}
{\bf (OPIs of Rota-Baxter Type)} For any $d, \lambda \in \bfk$, the
expressions $M(x,y)$ in the list below are of Rota-Baxter type
\paren{new types are underlined}. Moreover, any OPI of Rota-Baxter
type is necessarily of the form
$$\phi:= \lc x\rc \lc y\rc - \lc M(x,y)\rc,$$
for some $M(x,y)$ in the list.

\begin{enumerate}
\item $x\lc y\rc \quad$ \paren{average operator}, 
\item $\lc x\rc y \quad$ \paren{inverse average operator}, 
\item $\underline{x\lc y\rc +y\lc x\rc }$, 
\item $\underline{\lc x\rc y+\lc y\rc x}$, 
\item $x\lc y\rc +\lc x\rc y -\lc xy\rc \quad$ \paren{
Nijenhuis operator}, 
\item $x\lc y\rc +\lc x\rc y + \lambda xy \quad$
\paren{Rota-Baxter operator of weight $\lambda$}, 
\item $\underline{x\lc y\rc -x\lc 1\rc y + \lambda xy}$, 
\item $\underline{\lc x\rc y - x\lc 1\rc y + \lambda xy}$, 
\item $\underline{x\lc y\rc  + \lc x\rc y -x\lc 1\rc y+\lambda xy}
\quad$ \paren{generalized
Leroux TD operator with weight $\lambda$}, 
\item $\underline{x\lc y\rc +\lc x\rc y - xy\lc 1\rc -x\lc 1\rc y+
\lambda xy}$, 
\item $\underline{x\lc y\rc +\lc x\rc y -x\lc 1\rc y-\lc xy\rc +
\lambda xy}$, 
\item $\underline{x\lc y\rc +\lc x\rc y-x\lc 1\rc y-\lc 1\rc xy+
\lambda xy}$, 
\item $\underline{d x\lc 1\rc y + \lambda xy}\quad$ \paren{generalized
endomorphisms}, 
\item $\underline{d y\lc 1\rc x + \lambda yx} \quad$ \paren{generalized
antimorphisms}. 
\end{enumerate}
\mlabel{con:rbt}
\end{conjecture}

\begin{mremark} Let $\genset$ be any set. Recall that
the {bracketed words} in $\genset$ that are in RBRF when viewed as
elements of $\bfk\mapm{\genset}$ are called {\bf Rota-Baxter words}
and that they form a $\bfk$-basis of the free Rota-Baxter
$\bfk$-algebra on $\genset$. See~{\em \mcite{G-S}}. Every expression in
RBRF is a $\bfk$-linear combination of Rota-Baxter words in
$\bfk\mapm{x,y}$.

More generally, if $\phi(x,y)$ is of Rota-Baxter type, then the free
$\phi$-algebras on a set $\genset$ in the corresponding categories
of Rota-Baxter type $\phi$-algebras have special bases that can be
constructed uniformly. Indeed, if $\bfk\disj{\genset}$ denotes the
set of Rota-Baxter words in $\bfk\mapm{\genset}$, then the map
$$ \bfk\disj{\genset} \to
\bfk\mapm{\genset}\to \bfk\mapm{\genset}/I_\phi(Z)$$ is bijective.
Thus a suitable multiplication on $\bfk\disj{\genset}$ makes it the
free $\phi$-algebra on $\genset$. This is in fact how the free
Rota-Baxter $\bfk$-algebra on $\genset$ is constructed when
$\phi(x,y)$ is the OPI corresponding to the Rota-Baxter operator, the Nijenhuis operator~\cite{LG} and the TD operator~\cite{Zhou}.
\mlabel{rem:gen}
\end{mremark}

\section{Computational experiments}\mlabel{sec:comp}

In this section, we give a brief description of the computational
experiments in {\it Mathematica} that result in Conjectures
\mref{con:diffclass} and \mref{con:rbt}. The programs consist of
several Notebooks, available at \mcite{Sit} in a zipped file.

Basically, the non-commutative arithmetic for an operated algebra
was implemented {\it ad hoc}, using bracketed words and relying as
much as possible on the built-in facilities in {\it Mathematica} for
non-commutative multiplication, list operations, rewriting, and
equation simplification. Care was taken to avoid infinite recursions
during rewriting of expressions. An elaborate ansatz with
indeterminate coefficients (like the expression $N(x,y)$ in Theorem
\mref{thm:lowdeg}) is given as input, and to obtain differential
type OPIs, the difference $N(uv,w)-N(u,vw)$ is differentially
$\phi$-reduced using the rewrite rule system $\Sigma_\phi$. The
Rota-Baxter type OPIs are obtained similarly using an ansatz
$M(x,y)$ and reducing the difference $M(M(u,v),w)-M(u,M(v,w))$ with
the rewrite rule system $\Pi_\phi$. The resulting reduced form is
equated to zero, yielding a system of equations in the indeterminate
coefficients. This system is simplified using the method of Gr\"ober bases
(a heuristic application of Divide and Conquer has been automated).
Once the ansatz is entered, the ``algebras" can either be obtained in
one command {\tt getAlgebras}, or the computation can be stepped
through.

The programs provided 10 classes of differential type based on an
ansatz of 14 terms, which is then manually merged into the 6 classes
in Conjecture \mref{con:diffclass}. We obtain no new ones after
expanding the ansatz to 20 terms, including terms such as $\lc\lc x
\rc\rc \lc \lc y \rc\rc$.
The list for Rota-Baxter type OPIs are obtained
from an ansatz with 14 terms, some involving $P(1)$ (or $\lc \bfone
\rc$, in bracket notation) in a triple product.

We are quite confident that our list of differential type operators
is complete. For Rota-Baxter type operators, our list may not be
complete, since in our computations, we have restricted our
rewriting system $\Pi_\phi$ to disallow units in order to get around
the possibly non-terminating reduction sequences modulo the
identities. This is especially the case when the OPIs involve $\lc
\bfone \rc$. Typically, for Rota-Baxter type OPIs, we do not know how
to handle the appearance of $\lc\lc\bfone\rc\,\lc\bfone \rc\rc$
computationally (they may cancel, or not, if our rewriting system
$\Pi_\phi$ is expanded to include units as in Definition
\mref{de:rbtype}). While expressions involving $\lc \bfone \rc$
alone may be reduced to zero using an expanded rewriting system,
monomials involving a mix of bracketed words and $\lc \bfone \rc$
are often linearly independent over $\bfk$.

The {\it Mathematica} Notebook {\tt DTOrderTwoExamples.nb} shows the
computations for differential type operators and the Notebook {\tt
VariationRotaBaxterOperators.nb} does the same for Rota-Baxter type
ones. Non-commutative multiplication is printed using the symbol
$\otimes$ instead of $**$. It is known that the output routines fail
to be compatible with {\it Mathematica}, Version 8, and we will try
to fix this incompatibility and post updated versions on-line.

\section{Summary and outlook}
\mlabel{sec:sum}
We have studied Rota's classification problem by considering algebras with
a unary operator that satisfies operated polynomial identities. For
this, we have reviewed the construction of the operated polynomial algebra.

A far more general theory called variety of algebras exists, of
which the theories of PI-rings, PI-algebras, and OPI-algebras are
special cases \mcite{DF}. An ``algebra'' is any set with a set of
functions (operations), together with some identities perhaps. A
Galois connection between identities and ``variety of algebras'' is
set up similar to the correspondence between polynomial ideals and
algebraic varieties. Thus, differential algebra is one variety of
algebra, Rota-Baxter algebra is another, and so on.

In mathematics, specifically universal algebra \cite{BS,Cohn}, a
{\it {variety of algebras}} [or a {\it{finitary algebraic
category}}] is the class of all algebraic structures of a given
{signature} satisfying a given set of {identities}. Equivalently, a
variety is a class of algebraic structures of the same signature
which satisfies the HSP properties: {closed under the taking of
homomorphic images, subalgebras and (direct) products}.

This equivalence, known as the HSP Theorem, is a result of
G.\,Birkhoff, which is of fundamental importance in universal
algebra. We refer interested readers to \cite[Chap.\,I, Theorem
3.7]{Cohn}, \cite[Theorem 9.5]{BS} and, for computer scientists with
a model theory background, \cite[Theorem 3.5.14]{BN}. It is simple
to see that the class of algebras satisfying some set of equations
will be closed under the HSP operations. Proving the
converse---classes of algebras closed under the HSP operations must
be equational---is much harder.

By restricting ourselves to those special cases of Rota's Problem
that Rota was interested in, and by exploiting the structures of
operated algebra and compatibility of associativity on one hand, and
symbolic computation ({\it Mathematica}) on the other, we are able
to give two conjectured lists of OPI-algebras.

The project arose from our belief that the {construction of free
objects} in each class of the varieties should be uniformly done.
Currently, similar results for the known classes are proved
individually.

We also believe that there is a {Poincar\'e-Birkhoff-Witt} type
theorem, similar to the enveloping algebra of a Lie algebra, where a
canonical basis of the enveloping algebra is constructed from a
basis of the Lie algebra. Here, the free algebra of the variety is
constructed from the generating set $\genset$ with Rota-Baxter words
or terms (see Remark \mref{rem:gen}).

The theory of OPI-rings needs to be studied further and there are
many open problems. We end this discussion by providing just one. A
variety is {\bf Schreier} if every subalgebra of a free algebra in
the variety is free. For example, the variety of all groups (resp.
abelian groups) is Schreier. A central problem in the theory of
varieties is whether a particular variety of algebras is Schreier.
{Which of the varieties of differential type algebras or Rota-Baxter
type algebras are Schreier?}

%
%
\bibliographystyle{elsarticle-harv}
\hyphenpenalty=500

\end{document}